\documentclass[11pt]{article}
 
\usepackage{amsmath, amssymb,amsfonts,bbm,verbatim,multirow,color,xcolor}
\usepackage{epic,eepic,psfrag,epsfig}
\usepackage[sf,bf,SF,footnotesize]{subfigure}
\usepackage{graphicx}
\usepackage{amsthm,breakcites}
\usepackage{amscd}
\usepackage{epsfig}
\usepackage{fullpage}
\usepackage[round]{natbib}
\usepackage{setspace}
\usepackage{enumerate}
\usepackage{array}
\usepackage[small]{caption}
\RequirePackage[colorlinks,citecolor=blue,urlcolor=blue,breaklinks]{hyperref}

\newcommand{\red}[1]{\textcolor{red}{#1}}
\newcommand{\blue}[1]{\textcolor{blue}{#1}}

\newtheorem{thm}{Theorem}
\newtheorem{lemma}[thm]{Lemma}
\newtheorem{proposition}[thm]{Proposition}
\newtheorem{definition}{Definition}
\newtheorem{corollary}[thm]{Corollary}
\newtheorem{remark}{Remark}

\def\Cov{\textsf{Cov}} 
\def\Var{\mathsf{Var}} 

\def\text#1{\mbox{\rm #1}}

\newcommand{\lnorm}[2]{\|{#1} \|_{{#2}}}
\newcommand{\indc}[1]{{\mathbf{1}_{\left\{{#1}\right\}}}}
\newcommand{\norm}[1]{\|{#1} \|}

\newcommand{\wh}{\widehat}
\newcommand{\wt}{\widetilde}

\newcommand{\Fnorm}[1]{\lnorm{#1}{\rm F}}
\newcommand{\fnorm}[1]{\|#1\|_{\rm F}}
\newcommand{\opnorm}[1]{\|#1\|_{\rm op}}

\newcommand{\Prob}{\mathbb{P}}
\newcommand{\E}{\mathbb{E}}

\newcommand{\Tr}{\mathop{\sf Tr}}

\newcommand{\supp}{{\rm supp}}

\newcommand{\iprod}[2]{\left \langle #1, #2 \right\rangle}

\newcommand{\floor}[1]{{\left\lfloor {#1} \right \rfloor}}
\newcommand{\ceil}[1]{{\left\lceil {#1} \right \rceil}}
\newcommand{\TV}{{\sf TV}}

\newcommand{\brac}[1]{\left( #1\right)}

\title{Minimax rates in sparse, high-dimensional changepoint detection}

\date{}
\author{Haoyang Liu$^\ast$, Chao Gao$^\ast$ and Richard J. Samworth$^\dag$ \\ $^\ast$University of Chicago and $^\dag$University of Cambridge}

\begin{document}
\maketitle

\begin{abstract}
  We study the detection of a sparse change in a high-dimensional mean vector as a minimax testing problem.  Our first main contribution is to derive the exact minimax testing rate across all parameter regimes for $n$ independent, $p$-variate Gaussian observations.  This rate exhibits a phase transition when the sparsity level is of order $\sqrt{p \log \log (8n)}$ and has a very delicate dependence on the sample size: in a certain sparsity regime it involves a triple iterated logarithmic factor in~$n$.  Further, in a dense asymptotic regime, we identify the sharp leading constant, while in the corresponding sparse asymptotic regime, this constant is determined to within a factor of $\sqrt{2}$.  Extensions that cover spatial and temporal dependence, primarily in the dense case, are also provided. 
\end{abstract}


\section{Introduction}

The problem of changepoint detection has a long history \citep[e.g.][]{Page1955}, but has undergone a remarkable renaissance over the last 5--10 years.  This has been driven in part because these days sensors and other devices collect and store data on unprecedented scales, often at high frequency, which has placed a greater emphasis on the running time of changepoint detection algorithms \citep*{KFE2012,frick2014multiscale}.  But it is also because nowadays these data streams are often monitored simultaneously as a multidimensional process, with a changepoint in a subset of the coordinates representing an event of interest.  Examples include distributed denial of service attacks as detected by changes in traffic at certain internet routers \citep{PLR2004} and changes in a subset of blood oxygen level dependent contrast in a subset of voxels in fMRI studies \citep{AstonKirch2012}.  Away from time series contexts, the problem is also of interest, for instance in the detection of chromosomal copy number abnormality \citep{ZSJL2010,wang2018high}.  Key to the success of changepoint detection methods in such settings is the ability to borrow strength across the different coordinates, in order to be able to detect much smaller changes than would be possible through observation of any single coordinate in isolation.

We initially consider a simple model where, for some $n \geq 2$, we observe a $p \times n$ matrix $X$ that can be written as
\begin{equation}
\label{Eq:Model}
X = \theta + E,
\end{equation}
where $\theta \in \mathbb{R}^{p \times n}$ is deterministic and the entries of $E$ are independent $N(0,1)$ random variables.  We wish to test the null hypothesis that the columns of $\theta$ are constant against the alternative that there exists a time $t_0 \in \{1,\ldots,n-1\}$ at which these mean vectors change, in at most $s$ out of the $p$ coordinates.  The difficulty of this problem is governed by a signal strength parameter $\rho^2$ that measures the squared Euclidean norm of the difference between the mean vectors, rescaled by $\frac{t_0(n-t_0)}{n}$; this latter quantity can be interpreted as an effective sample size.  The goal is to identify the minimax testing rate in $\rho^2$ as a function of the problem parameters $p$, $n$ and $s$, and we denote this by $\rho^*(p,n,s)^2$; this is the signal strength at which we can find a test making the sum of the Type I and Type II error probabilities arbitrarily small by choosing $\rho^2$ to be an appropriately large multiple of $\rho^*(p,n,s)^2$ (where the multiple is not allowed to depend on $p$, $n$ and $s$), and at which any test has error probability sum arbitrarily close to $1$ for a suitably small multiple of $\rho^*(p,n,s)^2$.

Our first main contribution, in Theorem~\ref{thm:minimax}, is to reveal a particularly subtle form of the exact minimax testing rate in the above problem, namely 
\[
\rho^*(p,n,s)^2\asymp\begin{cases}
\sqrt{p\log\log(8n)} & \mbox{if $s\geq\sqrt{p\log\log(8n)}$}, \\
s\log\left(\frac{ep\log\log(8n)}{s^2}\right) \vee \log\log(8n) & \mbox{if $s<\sqrt{p\log\log(8n)}$.} \end{cases}
\]
This result provides a significant generalization of two known special cases in the literature, namely $\rho^*(1,n,1)^2$ and $\rho^*(p,2,s)^2$; see Section~\ref{SubSec:Special} for further discussion.  Although our initial optimal testing procedure depends on the sparsity level $s$, which would often be unknown in practice, we show in Theorem~\ref{thm:adaptive} that it is possible to construct an adaptive test that achieves exactly the same rate (but is a little more complicated to describe).

The theorem described above is a finite-sample result, but does not provide information at the level of constants.  By contrast, in Section~\ref{SubSec:Constants}, we study both dense and sparse asymptotic regimes, and identify the optimal constants exactly, in the former case, and to within a factor of $\sqrt{2}$ in the latter case.  In combination with Theorem~\ref{thm:minimax}, then, we are able to provide really quite a precise picture of the minimax testing rate in this problem.

Sections~\ref{sec:spatial} and~\ref{SECTEMP} concern extensions of our results to more general data generating mechanisms that allow for spatial and temporal dependence respectively.  In Section~\ref{sec:spatial}, we allow for cross-sectional dependence across the coordinates through a non-diagonal covariance matrix $\Sigma$ for the (Gaussian) columns of $E$.  We identify the sharp minimax testing rate when $s=p$, though the optimal procedure depends on three functionals of $\Sigma$, namely its trace, as well as its Frobenius and operator norms.  Estimation of these quantities is confounded by the potential presence of the changepoint, but we are able to propose a robust method that retains the same guarantee under a couple of additional conditions.  As an example, we consider covariance matrices that are a convex combination of the identity matrix and a matrix of ones; thus, each pair of distinct coordinates has the same (non-negative) covariance.  Interestingly, we find here that this covariance structure can make the problem either harder or easier, depending on the sparsity level of the changepoint.  In Section~\ref{SECTEMP}, we also focus on the case $s=p$ and allow dependence across the columns of $E$ (which are still assumed to be jointly Gaussian), controlled through a bound $B$ on the sum of the contributions of the operator norms of the off-diagonal blocks of the $np \times np$ covariance matrix.  Again, interesting phase transition phenomena in the testing rate occur here, depending on the relative magnitudes of the parameters $B$, $p$ and $n$. 

Most prior work on multivariate changepoint detection has proceeded without a sparsity condition and in an asymptotic regime with $n$ growing to infinity with the dimension fixed, including \citet{basseville1993detection}, \citet{CsorgoHorvath1997}, \citet{ombao2005slex}, \citet{aue2009break}, \citet{kirch2015detection}, \citet{ZSJL2010} and \citet{horvath2012change}.  \citet{bai2010common} studied the least squares estimator of a change in mean for high-dimensional panel data.  \citet{jirak2015uniform}, \citet{cho2015multiple}, \citet{cho2016change} and \citet{wang2018high} have all proposed CUSUM-based methods for the estimation of the location of a sparse, high-dimensional changepoint.   \citet{aston2018high} introduce a notion of efficiency that quantifies the detection power of different statistics in high-dimensional settings.  \citet{enikeeva2013high} study the sparse changepoint detection problem in an asymptotic regime in which $p \rightarrow \infty$, and at the same time $s \rightarrow \infty$ with $s/p \rightarrow \infty$ and the sample size not too large; we compare their results with ours in Section~\ref{Sec:Adaptation}.  Further related work on high-dimensional changepoint problems include the detection of changes in covariance \citep[e.g.][]{aue2009break,cribben2017estimating,wang2017optimal} and in sparse dynamic networks \citep{wang2018optimal}.  We emphasize that in this work we focus entirely on the offline version of the changepoint testing problem, where the entire data stream is observed prior to the statistician attempting to determine whether or not a change in mean has occurred.  For recent work on the corresponding online problem, where the data are observed sequentially and one wishes to declare a change as soon as possible after it has occured, see, e.g., \citet{xie2013sequential} and \citet{chen2020high}.


Proofs of our results in Sections~\ref{sec:main} and~\ref{sec:spatial} are given in Section~\ref{sec:proofs}, while proofs of the results in Section~\ref{SECTEMP} and various auxiliary lemmas are provided in the appendix. We close this section by introducing some notation that will be used throughout the paper. For $d \in \mathbb{N}$, we write $[d] := \{1,\ldots,d\}$.  Given $a,b\in\mathbb{R}$, we write $a\vee b:=\max(a,b)$ and $a\wedge b:=\min(a,b)$.  We also write $a\lesssim b$ to mean that there exists a universal constant $C > 0$ such that $a\leq Cb$; moreover, $a \asymp b$ means $a\lesssim b$ and $b\lesssim a$.  For a set $S$, we use $\mathbbm{1}_{S}$ and $|S|$ to denote its indicator function and cardinality respectively. For a vector $v = (v_1,\ldots,v_d)^T \in\mathbb{R}^d$, we define the norms $\norm{v}_1 := \sum_{\ell=1}^d |v_\ell|$, $\norm{v}^2:=\sum_{\ell=1}^d v_\ell^2$ and $\norm{v}_{\infty}:=\max_{\ell \in [d]}|v_\ell|$, and also define $\|v\|_0 := \sum_{\ell=1}^d \mathbbm{1}_{\{v_\ell \neq 0\}}$.  Given two vectors $u,v \in \mathbb{R}^d$ and a positive definite matrix $\Sigma \in \mathbb{R}^{d \times d}$, we define $\langle u,v \rangle_{\Sigma}^{-1} := u^T\Sigma^{-1}v$ and $\|v\|_{\Sigma^{-1}} := (v^T\Sigma^{-1}v)^{1/2}$ and omit the subscripts when $\Sigma = I_d$.  More generally, the trace inner product of two matrices $A,B\in\mathbb{R}^{d_1\times d_2}$ is defined as $\iprod{A}{B} :=\sum_{\ell=1}^{d_1}\sum_{\ell'=1}^{d_2}A_{\ell \ell'}B_{\ell \ell'}$, while the Frobenius and operator norms of $A$ are given by $\fnorm{A}:=\sqrt{\iprod{A}{A}}$ and $\opnorm{A}:=s_{\max}(A)$ respectively, where $s_{\max}(\cdot)$ denotes the largest singular value.  The total variation distance between two probability measures $P$ and $Q$ on a measurable space $(\mathcal{X},\mathcal{A})$ is defined as $\TV(P,Q):=\sup_{A \in \mathcal{A}}|P(A)-Q(A)|$.  Moreover, if $P$ is absolutely continuous with respect to $Q$, then the Kullback--Leibler divergence is defined as $D(P\|Q):=\int_{\mathcal{X}} \log\frac{dP}{dQ} \, dP$, and the chi-squared divergence is defined as $\chi^2(P\|Q):=\int_{\mathcal{X}} \bigl(\frac{dP}{dQ}-1\bigr)^2 \, dQ$.  The notation $\mathbb{P}$ and $\mathbb{E}$ are generic probability and expectation operators whose distribution is determined from the context.

\section{Main results}\label{sec:main}

Recall that we consider the observation of a $p \times n$ matrix $X=\theta+E$, where $n \geq 2$, where $\theta$ is deterministic and where each entry of the error matrix $E$ is an independent $N(0,1)$ random variable. In other words, writing $X_t$ and $\theta_t$ for the $t$th columns of $X$ and $\theta$ respectively, we have $X_t\sim N_p(\theta_t,I_p)$. The goal of our paper is to test whether or not the sequence $\{\theta_t\}_{t\in[n]}$ has a changepoint.  We define the parameter space of signals without a changepoint by
$$\Theta_0(p,n):=\left\{\theta\in\mathbb{R}^{p\times n}: \theta_t=\mu\text{ for some }\mu\in\mathbb{R}^p\text{ and all }t\in[n]\right\}.$$
For $s \in [p]$ and $\rho > 0$, the space consisting of signals with a sparse structural change at time $t_0 \in [n-1]$ is defined by
\begin{align*}
\Theta^{(t_0)}(p,n,s,\rho) &:= \Big\{\theta = (\theta_1,\ldots,\theta_n) \in\mathbb{R}^{p\times n}:\\ & \quad\quad\theta_t=\mu_1\text{ for some }\mu_1\in\mathbb{R}^p\text{ for all } 1\leq t\leq t_0, \\
& \quad\quad\theta_t=\mu_2\text{ for some }\mu_2\in\mathbb{R}^p\text{ for all }t_0+1\leq t\leq n, \\
& \quad\quad \|\mu_1-\mu_2\|_0\leq s, \frac{t_0(n-t_0)}{n}\|\mu_1-\mu_2\|^2\geq \rho^2\Big\}.
\end{align*}
In the definition of $\Theta^{(t_0)}(p,n,s,\rho)$, the parameters $p$ and $n$ determine the size of the problem, while $t_0$ is the location of the changepoint. The quantities $s$ and $\rho$ parametrize the sparsity level and the magnitude of the structural change respectively. It is worth noting that $\|\mu_1-\mu_2\|^2$ is normalized by the factor $\frac{t_0(n-t_0)}{n}$, which plays the role of the effective sample size of the problem. To understand this, consider the problem of testing the changepoint at location~$t_0$ when $p=1$. Then the natural test statistic is
$$\frac{1}{t_0}\sum_{t=1}^{t_0} X_t - \frac{1}{n-t_0}\sum_{t=t_0+1}^nX_t,$$
whose variance is $\frac{n}{t_0(n-t_0)}$.  Hence the difficulty of changepoint detection problem depends on the location of the changepoint.  Through the normalization factor $\frac{t_0(n-t_0)}{n}$, we can define a common signal strength parameter $\rho$ across different possible changepoint locations.  Taking a union over all such changepoint locations, the alternative hypothesis parameter space is given by
$$\Theta(p,n,s,\rho):=\bigcup_{t_0=1}^{n-1}\Theta^{(t_0)}(p,n,s,\rho).$$
We will address the problem of testing the two hypotheses
\begin{equation}
H_0:\theta\in\Theta_0(p,n),\qquad\qquad H_1:\theta\in\Theta(p,n,s,\rho). \label{eq:def-testing-problem}
\end{equation}
To this end, we let $\Psi$ denote the class of possible test statistics, i.e.~measurable functions $\psi:\mathbb{R}^{p \times n} \rightarrow [0,1]$.  We also define the minimax testing error by
$$\mathcal{R}(\rho) := \inf_{\psi \in \Psi}\biggl\{\sup_{\theta\in\Theta_0(p,n)}\mathbb{E}_{\theta}\psi(X) + \sup_{\theta\in\Theta(p,n,s,\rho)}\mathbb{E}_{\theta}\bigl(1-\psi(X)\bigr)\biggr\},$$
where we use $\mathbb{P}_{\theta}$ and $\mathbb{E}_{\theta}$ to denote probabilities and expectations under the data generating process~\eqref{Eq:Model}.  Our goal is to determine the order of the minimax rate of testing in this problem, as defined below.  
\begin{definition}\label{def:minimax-seperation}
	We say $\rho^*=\rho^*(p,n,s)$ is the minimax rate of testing if the following two conditions are satisfied:
	\begin{enumerate}
		\item For any $\epsilon\in(0,1)$, there exists $C_{\epsilon}>0$, depending only on $\epsilon$, such that $\mathcal{R}(C\rho^*)\leq\epsilon$ for any $C>C_{\epsilon}$.
		\item For any $\epsilon\in(0,1)$, there exists $c_{\epsilon}>0$, depending only on $\epsilon$, such that $\mathcal{R}(c\rho^*)\geq 1-\epsilon$ for any $c \in (0,c_{\epsilon})$.
	\end{enumerate}
\end{definition}

\subsection{Special cases}
\label{SubSec:Special}

Special cases of $\rho^*(p,n,s)$ are well understood in the literature.  For instance, when $p=s=1$, we recover the one-dimensional changepoint detection problem.  \cite{arias2011detection} showed that
\begin{equation}
\rho^*(1,n,1)^2 \asymp \log\log(8n).\label{eq:gao}
\end{equation}
The rate (\ref{eq:gao}) involves an iterated logarithmic factor, in constrast to a typical logarithmic factor in the minimax rate of sparse signal detection \citep[e.g.,][]{donoho2004higher,arias2005near,berthet2013optimal}.

Another solved special case is when $n=2$. In this setting, we observe $X_1\sim N_p(\mu_1,I_p)$ and $X_2\sim N_p(\mu_2,I_p)$, and the problem is to test whether or not $\mu_1=\mu_2$. Since $X_1-X_2$ is a sufficient statistic for $\mu_1-\mu_2$, the problem can be further reduced to a sparse signal detection problem in a Gaussian sequence model. For this problem, \cite{collier2017minimax} established the minimax detection boundary 
\begin{equation}
\rho^*(p,2,s)^2 \asymp \begin{cases}
\sqrt{p} & \mbox{if $s\geq \sqrt{p}$} \\
s\log\left(\frac{ep}{s^2}\right) & \mbox{if $s<\sqrt{p}$}.
\end{cases}\label{eq:collier}
\end{equation}
It is interesting to notice the elbow effect in the rate (\ref{eq:collier}).  Above the sparsity level of $\sqrt{p}$, one obtains the parametric rate that can be achieved using the test that rejects $H_0$ if $\|X_1-X_2\|_2^2 > 2p + c\sqrt{p}$ for an appropriate $c > 0$.

It is straightforward to extend both rates (\ref{eq:gao}) and (\ref{eq:collier}) to cases where either $p$ or $n$ is of a constant order. However, the general form of $\rho^*(p,n,s)$ is unknown in the statistical literature.

\subsection{Minimax detection boundary} \label{sec:m-d-b}

The main result of the paper is given by the following theorem.
\begin{thm}\label{thm:minimax}
	The minimax rate of the detection boundary of the problem (\ref{eq:def-testing-problem}) is given by
	\begin{equation}
	\rho^*(p,n,s)^2\asymp\begin{cases}
	\sqrt{p\log\log(8n)} & \mbox{if $s\geq\sqrt{p\log\log(8n)}$} \\
	s\log\left(\frac{ep\log\log(8n)}{s^2}\right) \vee \log\log(8n) & \mbox{if $s<\sqrt{p\log\log(8n)}.$}
	\end{cases}\label{eq:minimax-rate}
	\end{equation}
\end{thm}

It is important to note that the minimax rate (\ref{eq:minimax-rate}) is not a simple sum or multiplication of the rates (\ref{eq:gao}) and (\ref{eq:collier}) for constant $p$ or $n$. The high-dimensional changepoint detection problem differs fundamentally from both its low-dimensional version and the sparse signal detection problem.

We observe that the minimax rate exhibits the two regimes in~\eqref{eq:minimax-rate} only when $p\geq\log\log(8n)$, since if $p<\log\log(8n)$, then the condition $s\geq\sqrt{p\log\log(8n)}$ is empty, and~\eqref{eq:minimax-rate} has just one regime.  Compared with the rate (\ref{eq:collier}), the phase transition boundary for the sparsity $s$ becomes $\sqrt{p\log\log(8n)}$. In fact, the minimax rate (\ref{eq:minimax-rate}) can be obtained by first replacing the $p$ in (\ref{eq:collier}) with $p\log\log(8n)$, and then adding the extra term (\ref{eq:gao}).

The dependence of (\ref{eq:minimax-rate}) on $n$ is very delicate. Consider the range of sparsity where
$$\frac{\log\log(8n)}{\log(e\log\log(8n))} \vee \frac{\sqrt{p}}{(\log\log(8n))^C}\lesssim s\lesssim \sqrt{p\log\log(8n)},$$
for some universal constant $C>0$. The rate (\ref{eq:minimax-rate}) then becomes
$$\rho^*(p,n,s)^2 \asymp s\log(e\log\log(8n)).$$
That is, it grows with $n$ at a $\log\log\log(\cdot)$ rate. To the best of our knowledge, such a triple iterated logarithmic rate has not been found in any other problem before in the statistical literature.

Last but not least, we remark that when $p$ or $n$ is a constant, the rate (\ref{eq:minimax-rate}) recovers (\ref{eq:gao}) and (\ref{eq:collier}) as special cases.

\subsubsection{Upper Bound}

To derive the upper bound, we need to construct a testing procedure. We emphasize that the goal of hypothesis testing is to detect the existence of a changepoint; this is in contrast to the problem of changepoint estimation \citep{cho2015multiple,wang2018high,wang2018univariate}, where the goal is to find the changepoint's location.

If we knew that the changepoint were between $t$ and $n-t+1$, it would be natural to define the Cumulative Sum (CUSUM)-type statistic
\begin{equation}
Y_t:=\frac{(X_1+ \ldots +X_t)-(X_{n-t+1}+ \ldots +X_n)}{\sqrt{2t}}.\label{eq:Y_t}
\end{equation}
Note that the definition of $Y_t$ does not use the observations between $t+1$ and $n-t$. This allows $Y_t$ to detect any changepoint in this range, regardless of its location.  The existence of a changepoint implies that $\mathbb{E}_\theta(Y_t) \neq 0$. Since the structural change only occurs in a sparse set of coordinates, we threshold the magnitude of each coordinate $Y_t(j)$ at level $a \geq 0$ to obtain
$$A_{t,a}:=\sum_{j=1}^p \bigl\{Y_t(j)^2-\nu_a\bigr\}\mathbbm{1}_{\{|Y_t(j)|\geq a\}},$$
where $\nu_a:=\mathbb{E}\bigl(Z^2\bigm||Z| \geq a\bigr)$ is the conditional second moment of $Z \sim N(0,1)$, given that its magnitude is at least $a$. See \cite{collier2017minimax} for a similar strategy for the sparse signal detection problem. Note that $A_{t,0}=\sum_{j=1}^p \bigl\{Y_t(j)^2-1\bigr\}$ has a centered $\chi_p^2$ distribution under~$H_0$.

Since the range of the potential changepoint locations is unknown, a natural first thought is to take a maximum of $A_{t,a}$ over $t \in [n/2]$.  It turns out, however, that in high-dimensional settings it is very difficult to control the dependence between these different test statistics at the level of precision required to establish the minimax testing rate.  A methodological contribution of this work, then, is the recognition that it suffices to compute a maximum of $A_{t,a}$ over a candidate set $\mathcal{T}$ of locations, because if there exists a changepoint at time $t_0$ and $t_0/2 < \wt{t} \leq t_0$ for some $\wt{t} \in \mathcal{T}$, then $\|\mathbb{E}_{\theta}(Y_{\wt{t}})\|$ and $\|\mathbb{E}_{\theta}(Y_{t_0})\|$ are of the same order of magnitude.  This observation reflects a key difference between the changepoint testing and estimation problems.  To this end, we define
$$\mathcal{T}:=\left\{1,2,4,\ldots,2^{\floor{\log_2(n/2)}}\right\},$$
so that $|\mathcal{T}| = 1 + \floor{\log_2(n/2)}$.  Then, for a given $r \geq 0$, the testing procedure we consider is given by
\begin{equation}
\psi \equiv \psi_{a,r}(X) :=\mathbbm{1}_{\{\max_{t\in\mathcal{T}}A_{t,a}>r\}}. \label{eq:def-test}
\end{equation}
The theoretical performance of the test~\eqref{eq:def-test} is given by the following theorem. We use the notation $r^*(p,n,s)$ for the rate function on the right-hand side of (\ref{eq:minimax-rate}).
\begin{proposition}\label{thm:upper}
	For any $\epsilon\in(0,1)$, there exists $C>0$, depending only on $\epsilon$, such that the testing procedure~(\ref{eq:def-test}) with $a^2=4\log\left(\frac{ep\log\log(8n)}{s^2}\right)\mathbbm{1}_{\{s<\sqrt{p\log\log(8n)}\}}$ and $r=Cr^*(p,n,s)$ satisfies
	$$\sup_{\theta\in\Theta_0(p,n)}\mathbb{E}_{\theta}\psi + \sup_{\theta\in\Theta(p,n,s,\rho)}\mathbb{E}_{\theta}(1-\psi)\leq\epsilon,$$
	as long as $\rho^2\geq 32Cr^*(p,n,s)$.
\end{proposition}
Just as the minimax rate (\ref{eq:minimax-rate}) has two regimes, the testing procedure (\ref{eq:def-test}) also uses two different strategies. In the dense regime $s\geq\sqrt{p\log\log(8n)}$, we have $a^2=0$ and thus (\ref{eq:def-test}) becomes simply $\psi=\mathbbm{1}_{\{\max_{t\in\mathcal{T}}\|Y_t\|^2-p > r\}}$. In the sparse regime $s<\sqrt{p\log\log(8n)}$, a thresholding rule is applied at level $a$, where $a^2=4\log\left(\frac{ep\log\log(8n)}{s^2}\right)$.  We discuss adaptivity to the sparsity level $s$ in Section~\ref{Sec:Adaptation}.

\subsubsection{Lower Bound} We show that the testing procedure (\ref{eq:def-test}) is minimax optimal by stating a matching lower bound. 
\begin{proposition}\label{thm:lower}
	For any $\epsilon\in(0,1)$, there exists $c>0$, depending only on $\epsilon$, such that $\mathcal{R}(\rho)\geq 1- \epsilon$ whenever $\rho^2\leq cr^*(p,n,s)$.
\end{proposition}

\subsection{Adaptation to sparsity}
\label{Sec:Adaptation}

The testing procedure (\ref{eq:def-test}) that achieves the minimax detection rate depends on knowledge of the sparsity $s$. In this section, we present an alternative procedure that is adaptive to $s$.  The idea is to take supremum over a grid of sparsity levels. Recall the definition of the testing procedure $\psi_{a,r}$ in (\ref{eq:def-test}), and let us makes the dependence on~$s$ explicit by writing
$$\psi^{(s)}:=\psi_{a(s),r(s)},$$
where $a^2(s):=4\log\left(\frac{ep\log\log(8n)}{s^2}\right)\mathbbm{1}_{\{s<\sqrt{p\log\log(8n)}\}}$ and $r(s):=Cr^*(p,n,s)$ as in Proposition \ref{thm:upper}. Then our adaptive test is defined by
$$\psi_{\rm adaptive}:=\max_{s\in\mathcal{S}}\psi^{(s)},$$
where
$$\mathcal{S}:=\Bigl\{1,2,4,\ldots,2^{\ceil{\log_2\bigl(\sqrt{p \log \log (8n)}\bigr)}-1}\Bigr\} \cup \{p\}.$$
The choice of this particular grid for $\mathcal{S}$ is not essential (we could also take $\mathcal{S} := [p]$), but it reduces computation.
\begin{thm}\label{thm:adaptive}
	For any $\epsilon\in(0,1)$, there exists $C>0$, depending only on $\epsilon$, such that the testing procedure $\psi_{\rm adaptive}$ satisfies
	$$\sup_{\theta\in\Theta_0(p,n)}\mathbb{E}_{\theta}\psi_{\rm adaptive} + \sup_{\theta\in\Theta(p,n,s,\rho)}\mathbb{E}_{\theta}(1-\psi_{\rm adaptive})\leq\epsilon,$$
	as long as $\rho^2\geq 64Cr^*(p,n,s)$. 
\end{thm}

Theorem \ref{thm:adaptive} shows that the minimax detection boundary (\ref{thm:minimax}) can be achieved adaptively without the knowledge of the sparsity level $s$. In the literature, changepoint detection with unknown sparsity was also investigated by \cite{enikeeva2013high}. Their procedure has a vanishing testing error as long as
\begin{equation}
\rho^2\gtrsim \min\left(\sqrt{p\log p}+\sqrt{p\log\log n}, s\log\frac{p}{s}\right),\label{eq:EH-rate}
\end{equation}
under the additional assumptions that $p,s\rightarrow\infty$, $s/p\rightarrow 0$, and $\frac{\log n}{s\log(p/s)}\rightarrow 0$. Comparing (\ref{eq:EH-rate}) with the optimal rate (\ref{thm:minimax}), we see that \cite{enikeeva2013high} successfully identified the $\sqrt{p\log\log n}$ term in the dense regime and the $s\log p$ term in the sparse regime. However, we can also observe that the $\sqrt{p\log p}$ term is not necessary, and the rate (\ref{eq:EH-rate}) is in general not sharp without the assumption $\frac{\log n}{s\log(p/s)}\rightarrow 0$, especially when the sparsity level $s$ is around $\sqrt{p\log\log n}$. 

\subsection{Asymptotic constants}
\label{SubSec:Constants}

A notable feature of our minimax detection boundary derived in Theorem \ref{thm:minimax} is that the rate is non-asymptotic, meaning that the result holds for arbitrary $n\geq 2$, $p \in \mathbb{N}$ and $s \in [p]$.  On the other hand, if we are allowed to make a few asymptotic assumptions, we can give explicit constants for the lower and upper bounds.  In this subsection, therefore, we let both the dimension $p$ and the sparsity $s$ be functions of $n$, and we consider asymptotics as $n\rightarrow\infty$.

\begin{thm}[Dense regime]\label{thm:asymp-dense}
	Assume that $s^2/(p\log\log n)\rightarrow\infty$ as $n\rightarrow\infty$. Then, with
	$$\rho=\xi\left(p\log\log n\right)^{1/4},$$
	we have $\mathcal{R}(\rho)\rightarrow 0$ when $\xi>\sqrt{2}$ and $\mathcal{R}(\rho)\rightarrow 1$ when $\xi<\sqrt{2}$.
\end{thm}

\begin{thm}[Sparse regime]\label{thm:asymp-sparse}
	Assume that $s^2/p\rightarrow 0$ and $s/\log\log n \rightarrow\infty$ as $n\rightarrow\infty$. Then, with
	$$\rho=\xi \sqrt{s\log\left(\frac{p\log\log n}{s^2}\right)},$$
	we have $\mathcal{R}(\rho)\rightarrow 0$ when $\xi>\sqrt{2}$ and $\mathcal{R}(\rho)\rightarrow 1$ when $\xi<1$.
\end{thm}
These two theorems characterize the asymptotic minimax upper and lower bounds of the changepoint detection problem under dense and sparse asymptotics respectively.  While we are able to nail down the exact asymptotic constant in the dense regime, the optimal asymptotic constant in the sparse regime is a more involved problem (indeed, it appears to depend on more refined aspects of the asymptotic regime), and we therefore leave it as an open problem for future research.

\section{Spatial dependence}\label{sec:spatial}

In this section, we consider changepoint detection in settings with cross-sectional dependence in the $p$ coordinates. To be specific, we now relax our previous assumption on the cross-sectional distribution by supposing only that $X_t\sim N_p(\theta_t,\Sigma)$ for some general positive definite covariance matrix $\Sigma \in \mathbb{R}^{p \times p}$; the goal remains to solve the testing problem~\eqref{eq:def-testing-problem}. We retain the notation $\mathbb{P}_{\theta}$ and $\mathbb{E}_{\theta}$ for probabilities and expectations, with the dependence on $\Sigma$ suppressed.  Similar to Definition \ref{def:minimax-seperation}, we use the notation $\rho^*_{\Sigma}(p,n,s)$ for the minimax rate of testing in this problem with cross-sectional covariance $\Sigma$.

Our first result provides the minimax rate of the detection boundary in the dense case where $s=p$. This sets up a useful benchmark on the difficulty of the problem depending on the covariance structure. 
\begin{thm}\label{thm:minimax-spatial}
	In the case $s=p$, the minimax rate of testing is given by
	\begin{equation}
	\rho_{\Sigma}^*(p,n,p)^2\asymp\fnorm{\Sigma}\sqrt{\log\log(8n)} \vee \opnorm{\Sigma}\log\log(8n). \label{eq:minimax-spatial-dense}
	\end{equation}
\end{thm}
For $\Sigma=I_p$, Theorem~\ref{thm:minimax-spatial} yields $\rho_{\Sigma}^*(p,n,p)^2\asymp\sqrt{p\log\log(8n)}\vee\log\log(8n)$, which recovers the result of Theorem~\ref{thm:minimax} when $s=p$. The more general result for $s<p$ with a non-diagonal~$\Sigma$ is hard to obtain. This is because the proof of Theorem \ref{thm:minimax-spatial} relies on a diagonalization argument of the covariance matrix, which can affect the sparsity pattern of the change in the mean unless additional assumptions on $\Sigma$ similar to \cite{hall2010innovated} are imposed.

A test that achieves the optimal rate~\eqref{eq:minimax-spatial-dense} is given by
\begin{equation}
\psi := \mathbbm{1}_{\left\{\max_{t\in\mathcal{T}}\|Y_t\|^2-\Tr(\Sigma) > C\left(\fnorm{\Sigma}\sqrt{\log\log(8n)} \vee \opnorm{\Sigma}\log\log(8n)\right)\right\}}, \label{eq:def-test-spatial-cov-known}
\end{equation}
for an appropriate choice of $C > 0$.  Though optimal, the procedure (\ref{eq:def-test-spatial-cov-known}) relies on knowledge of $\Sigma$. In fact, one only needs to know $\Tr(\Sigma)$, $\fnorm{\Sigma}$ and $\opnorm{\Sigma}$, rather than the entire covariance matrix~$\Sigma$.  To be even more specific, from a careful examination of the proof, we see that we only need to know $\Tr(\Sigma)$ up to an additive error that is at most of the same order as the cut-off, whereas knowledge of the orders of $\fnorm{\Sigma}$ and $\opnorm{\Sigma}$, up to multiplication by universal constants, is enough.

We now discuss how to use $X$ to estimate the three quantities $\Tr(\Sigma), \fnorm{\Sigma}$ and $\opnorm{\Sigma}$. The solution would be straightforward if we knew the location of the changepoint, but in more typical situations where the changepoint location is unknown, this becomes a robust covariance functional estimation problem. 
We assume that $n\geq 6$ and that $n/3$ is an integer, since a simple modification can be made if $n/3$ is not a integer. We can then divide $[n]$ into three consecutive blocks $\mathcal{D}_1,\mathcal{D}_2,\mathcal{D}_3$, each of whose cardinalities is $n/3\geq 2$.  For $j \in [3]$, we compute the sample covariance matrix
$$\wh{\Sigma}_{\mathcal{D}_j}:=\frac{1}{|\mathcal{D}_j|-1}\sum_{t\in\mathcal{D}_j}(X_t-\bar{X}_{\mathcal{D}_j})(X_t-\bar{X}_{\mathcal{D}_j})^T,$$
where $\bar{X}_{\mathcal{D}_j} := |\mathcal{D}_j|^{-1}\sum_{t \in \mathcal{D}_j} X_t$.  We can then order these three estimators according to their trace, and Frobenius and operator norms, yielding
\begin{eqnarray*}
	&&\Tr(\wh{\Sigma})^{(1)} \leq \Tr(\wh{\Sigma})^{(2)} \leq \Tr(\wh{\Sigma})^{(3)}, \\
	&&\fnorm{\wh{\Sigma}}^{(1)} \leq \fnorm{\wh{\Sigma}}^{(2)} \leq \fnorm{\wh{\Sigma}}^{(3)}, \\
	&&\opnorm{\wh{\Sigma}}^{(1)} \leq \opnorm{\wh{\Sigma}}^{(2)} \leq \opnorm{\wh{\Sigma}}^{(3)}.
\end{eqnarray*}
The idea is that at least two of the three covariance matrix estimators $\wh{\Sigma}_{\mathcal{D}_1}, \wh{\Sigma}_{\mathcal{D}_2}, \wh{\Sigma}_{\mathcal{D}_3}$ should be accurate, because there is at most one changepoint location. This motivates us to take the medians $\Tr(\wh{\Sigma})^{(2)}$, $\fnorm{\wh{\Sigma}}^{(2)}$ and $\opnorm{\wh{\Sigma}}^{(2)}$ with respect to the three functionals as our robust estimators.  It is convenient to define $\Theta(p,n,s,0) := \Theta_0(p,n) \cup \left(\cup_{\rho > 0} \Theta(p,n,s,\rho)\right)$.
\begin{proposition}\label{prop:est-cov}
	Assume $p \leq cn$ for some $c > 0$, and fix an arbitrary positive definite $\Sigma \in \mathbb{R}^{p \times p}$ and $\theta\in\Theta(p,n,p,0)$.  Then given $\epsilon > 0$, there exists $C>0$, depending only on $c$ and $\epsilon$, such that 
	\begin{align*}
	\bigl|\Tr(\wh{\Sigma})^{(2)}-\Tr(\Sigma)\bigr| &\leq C\left(\frac{\sqrt{p}\fnorm{\Sigma}}{\sqrt{n}}+\frac{p\opnorm{\Sigma}}{n}\right), \\
	\bigl|\fnorm{\wh{\Sigma}}^{(2)}-\fnorm{\Sigma}\bigr| &\leq C\opnorm{\Sigma}\sqrt{\frac{p^2}{n}}\\
	\bigl|\opnorm{\wh{\Sigma}}^{(2)}-\opnorm{\Sigma}\bigr| &\leq C\opnorm{\Sigma}\sqrt{\frac{p}{n}},
	\end{align*}
	with $\mathbb{P}_\theta$-probability at least $1-\epsilon/4$.
\end{proposition}
With the help of Proposition~\ref{prop:est-cov}, we can plug the estimators $\Tr(\wh{\Sigma})^{(2)}, \fnorm{\wh{\Sigma}}^{(2)}$ and $\opnorm{\wh{\Sigma}}^{(2)}$ into the procedure (\ref{eq:def-test-spatial-cov-known}). This test is adaptive to the unknown covariance structure, and comes with the following performance guarantee.

\begin{corollary}\label{cor:upper-cov-spatial}
	Assume that $\sqrt{p}\opnorm{\Sigma}\leq A\fnorm{\Sigma}$ for some $A > 0$.  Then given $\epsilon > 0$, there exist $c,C > 0$, depending only on $A$ and $\epsilon$, such that if $p \leq cn$, then the testing procedure
	$$\psi_{\mathrm{Cov}} := \mathbbm{1}_{\{\max_{t\in\mathcal{T}}\|Y_t\|^2-\Tr(\wh{\Sigma})^{(2)} > C(\fnorm{\wh{\Sigma}}^{(2)}\sqrt{\log\log(8n)} \vee \opnorm{\wh{\Sigma}}^{(2)}\log\log(8n))\}}$$
	satisfies
	$$\sup_{\theta\in\Theta_0(p,n)}\mathbb{E}_{\theta}\psi_{\mathrm{Cov}} + \sup_{\theta\in\Theta(p,n,p,\rho)}\mathbb{E}_{\theta}(1-\psi_{\mathrm{Cov}})\leq\epsilon,$$
	as long as $\rho^2\geq 64C\left(\fnorm{\Sigma}\sqrt{\log\log(8n)}\vee\opnorm{\Sigma}\log\log(8n)\right)$.
\end{corollary}
\begin{remark}
	The conditions $p\lesssim n$ and $\sqrt{p}\opnorm{\Sigma}\lesssim \fnorm{\Sigma}$ guarantee that $\bigl|\fnorm{\wh{\Sigma}}^{(2)}-\fnorm{\Sigma}\bigr|\lesssim \fnorm{\Sigma}$ and $\bigl|\opnorm{\wh{\Sigma}}^{(2)}-\opnorm{\Sigma}\bigr|\lesssim \opnorm{\Sigma}$ with high probability, by Proposition~\ref{prop:est-cov}. Note that $\sqrt{p}\opnorm{\Sigma}\lesssim \fnorm{\Sigma}$ will be satisfied if all eigenvalues of $\Sigma$ are of the same order.  In fact, it is possible to weaken the condition $\sqrt{p}\opnorm{\Sigma}\lesssim \fnorm{\Sigma}$ using the notion of effective rank \citep{koltchinskii2017concentration}; however, this greatly complicates the analysis, and we do not pursue this here.  Alternatively, Corollary~\ref{cor:upper-cov-spatial} also holds without the $\sqrt{p}\opnorm{\Sigma}\leq A\fnorm{\Sigma}$ condition but under the stronger dimensionality restriction $p^2\leq cn$; this then allows for an arbitrary covariance matrix $\Sigma$.
\end{remark}
To better understand the influence of the covariance structure, consider, for $\gamma \in [0,1)$, the covariance matrix
$$\Sigma(\gamma):=(1-\gamma)I_p + \gamma \mathbf{1}_p\mathbf{1}_p^T,$$
which has diagonal entries $1$ and off-diagonal entries $\gamma$. The parameter $\gamma$ controls the pairwise spatial dependence; moreover, $\fnorm{\Sigma(\gamma)}^2 = (1-\gamma^2)p + p^2\gamma^2$ and $\opnorm{\Sigma(\gamma)}=1+(p-1)\gamma$. By Theorem~\ref{thm:minimax-spatial}, we have
\begin{equation}
\rho^*_{\Sigma(\gamma)}(p,n,p)^2 \asymp \sqrt{\bigl\{(1-\gamma^2)p + p^2\gamma^2\bigr\}\log\log(8n)} \vee \bigl\{1+(p-1)\gamma\bigr\}\log\log(8n). \label{eq:correlation-adds-diff}
\end{equation}
Thus the spatial dependence significantly increases the difficulty of the testing problem. In particular, if $\gamma$ is of a constant order, then the minimax rate is $p\log\log(8n)$, which is much larger than the rate~\eqref{eq:minimax-spatial-dense} for $\Sigma=I_p$.

However, the increased difficulty of testing in this example is just one part of the story. When we consider the sparsity factor $s$, the influence of the covariance structure can be the other way around. To illustrate this interesting phenomenon, we discuss a situation where $s$ is small. Since $X_t\sim N_p\bigl(\theta_t,\Sigma(\gamma)\bigr)$, we have that $Y_t\sim N_p\bigl(\Delta_t,\Sigma(\gamma)\bigr)$ for $t < n/2$, where $\Delta_t := \frac{(\theta_1+\ldots +\theta_t)-(\theta_{n-t+1}+ \ldots +\theta_n)}{\sqrt{2t}}$. Hence, the distribution of $Y_t$ can be expressed in terms of a factor model. That is,
\begin{equation}
\label{Eq:Ytj}
Y_t(j) = \Delta_t(j) + \sqrt{\gamma}W_t + \sqrt{1-\gamma}Z_{tj},
\end{equation}
where $W_{t},Z_{t1},\ldots,Z_{tp}\stackrel{\mathrm{iid}}{\sim} N(0,1)$. When there is no changepoint, we have $\Delta_t=0$, so $Y_t(j)|W_t\stackrel{\mathrm{iid}}{\sim} N(\sqrt{\gamma}W_t,1-\gamma)$ for all $j\in[p]$. When there is a changepoint between $t$ and $n-t+1$, we have $\|\Delta_t\|_0 \leq s$. In either case, then, we can estimate $\sqrt{\gamma}W_t$ by $\text{\sf Median}(Y_t)$. This motivates the new statistic
\begin{equation}
\wt{Y}_t:=\frac{Y_t - \text{\sf Median}(Y_t)\mathbf{1}_p}{\sqrt{1-\gamma}}.\label{eq:def-tilde-Y}
\end{equation}
To construct a scalar summary of $\wt{Y}_t$, we define the functions $f_a(x):=(x^2-\nu_a)\mathbbm{1}_{\{|x| \geq a\}}$ for $x \in \mathbb{R}$ and, for $C' \geq 0$, set
\begin{equation}
g_a(x) \equiv g_{a,C'}(x) :=\inf\biggl\{f_a(y): |y-x|\leq C'\sqrt{\frac{\log\log(8n)}{p}}\biggr\}. \label{eq:def-g-a}
\end{equation}
Note that $g_a(x) = f_a(x)$ when $C' = 0$.  The use of a positive $C'>0$ in (\ref{eq:def-g-a}) is to tolerate the error of $\text{\sf Median}(Y_t)$ as an estimator of $\sqrt{\gamma}W_t$.  The new testing procedure is then
\begin{equation}
\psi_{a,r,C'} := \mathbbm{1}_{\{\max_{t\in\mathcal{T}}\sum_{j=1}^pg_a(\wt{Y}_t(j))>r\}}. \label{eq:def-test-gamma}
\end{equation}
\begin{thm}\label{thm:sparse-upper-spatial}
	Assume that $\gamma\in[0,1)$ and $s\leq (p\log\log(8n))^{1/5}$.
	Then there exist universal constants $c,C' > 0$ such that if $\frac{\log\log(8n)}{p}\leq c$, then for any $\epsilon\in(0,1)$, we can find $C>0$ and $n_0 \in \mathbb{N}$, both depending only on $\epsilon$, such that the testing procedure (\ref{eq:def-test-gamma}) with $a^2=4\log\left(\frac{ep\log\log(8n)}{s^2}\right)$ and $r=C(1-\gamma)\left(s\log\left(\frac{ep\log\log(8n)}{s^2}\right)\vee\log\log(8n)\right)$ satisfies $$\sup_{\theta\in\Theta_0(p,n)}\mathbb{E}_{\theta}\psi_{a,r,C'} + \sup_{\theta\in\Theta(p,n,s,\rho)}\mathbb{E}_{\theta}(1-\psi_{a,r,C'})\leq\epsilon,$$
	for $n \geq n_0$, provided $\rho^2\geq 32C(1-\gamma)\bigl\{s\log\bigl(eps^{-2}\log\log(8n)\bigr)\vee\log\log(8n)\bigr\}$.
\end{thm}

Surprisingly, in the sparse regime, the spatial correlation helps changepoint detection, and the required signal strength for testing consistency decreases as $\gamma$ increases. This is in stark contrast to (\ref{eq:correlation-adds-diff}) for the same covariance structure when $s=p$.

\begin{remark}
	The testing procedure considered in Theorem \ref{thm:sparse-upper-spatial} can be easily made adaptive to the unknown $\gamma$ by taking advantage of Proposition \ref{prop:est-cov}. Since $\Tr(\Sigma(\gamma))=p+(p^2-p)\gamma$, when $p \geq 2$ the estimator $\wh{\gamma}:=\frac{\Tr(\wh{\Sigma})^{(2)}-p}{p^2-p}$ satisfies
	$$|\wh{\gamma}-\gamma|\lesssim \frac{\sqrt{(1-\gamma^2)p+p^2\gamma^2}}{p^{3/2}\sqrt{n}}+\frac{1+(p-1)\gamma}{pn},$$
	with probability at least $1-2e^{-p}$. Then, the procedure with $\gamma$ replaced by $\wh{\gamma}$ enjoys the same guarantee of Theorem \ref{thm:sparse-upper-spatial} under mild extra conditions.
\end{remark}

The next theorem shows that the rate achieved by Theorem \ref{thm:sparse-upper-spatial} is minimax optimal.  
\begin{thm}\label{thm:sparse-lower-spatial}
	Assume that $\gamma\in[0,1)$ and $s\leq\sqrt{p\log\log n}$. Then
	\begin{equation}
	\label{Eq:rstargamma}
	\rho_{\Sigma(\gamma)}^*(p,n,s)^2\gtrsim (1-\gamma)\biggl\{s\log\left(\frac{ep\log\log(8n)}{s^2}\right) \vee \log\log(8n)\biggr\}.
	\end{equation}
\end{thm}
To conclude this section, we remark that the dependence on $\gamma$ of the minimax testing rate arises in part due to our choice of measuring departures from the null hypothesis in terms of a rescaled squared Euclidean distance.  Other natural choices, such as a rescaled squared supremum norm distance \citep{jirak2015uniform} may well lead to different phenomena.

\section{Temporal dependence}\label{SECTEMP}

In this section, we consider the situation where $X_1,\ldots,X_n$ form a multivariate time series. To be specific, in our model $X_t=\theta_t+E_t$ for $t\in[n]$, we now assume that the random vectors $E_1,\ldots,E_n$ are jointly Gaussian but not necessarily independent. The covariance structure of the error vectors can be parametrized by a covariance matrix $\Sigma\in\mathbb{R}^{pn\times pn}$, and for $B \geq 0$, we write $\Sigma \in \mathcal{C}(p,n,B)$ if: 
\begin{enumerate}
	\item $\Cov(E_t)=I_p$ for all $t\in [n]$;
	\item  $\sum_{s\in[n]\setminus\{t\}}\opnorm{\Cov(E_s,E_t)}\leq B$ for all $t\in[n]$.
\end{enumerate}
Thus the data generating process of $X$ is completely determined by its mean matrix $\theta$ and covariance matrix $\Sigma\in\mathcal{C}(p,n,B)$, and we use the notion $\mathbb{P}_{\theta,\Sigma}$ and $\mathbb{E}_{\theta,\Sigma}$ for the corresponding probability and expectation.  The case $B=0$ reduces to the situation of observations at different time points being independent. Time series dependence in high-dimensional changepoint problems has also been considered by \cite{wang2018high}; their condition $\opnorm{\sum_{s=1}^n\Cov(E_s,E_t)}\leq B$ for all $t\in[n]$ is only slightly different from ours.  We also mention here the work of \cite{horvath2012change}, who study the asymptotic distributions of changepoint test statistics with dependent data, in a regime in which $p/\sqrt{n} \rightarrow 0$. 

We focus on the case $s=p$ and do not consider the effect of sparsity.  The minimax testing error is defined by
$$\mathcal{R}(\rho) := \inf_{\psi \in \Psi}\Biggl\{\sup_{\substack{\theta\in\Theta_0(p,n)\\\Sigma\in\mathcal{C}(p,n,B)}}\mathbb{E}_{\theta,\Sigma}\psi + \sup_{\substack{\theta\in\Theta(p,n,p,\rho)\\\Sigma\in\mathcal{C}(p,n,B)}}\mathbb{E}_{\theta,\Sigma}(1-\psi)\Biggr\}.$$
We also define the corresponding minimax rate of detection boundary $\rho_{\mathrm{Temp}}^*(p,n,p,B)$ similarly to Definition~\ref{def:minimax-seperation}.
The testing procedure
\begin{equation}
\psi_{\mathrm{Temp}}: =\mathbbm{1}_{\{\max_{t\in\mathcal{T}}\|Y_t\|^2-p>r\}}\label{eq:def-test-temporal}
\end{equation}
has the following property:
\begin{thm}\label{thm:upper-temporal}
	For any $\epsilon\in(0,1)$, there exists $C>0$, depending only on $\epsilon$, such that the test~\eqref{eq:def-test-temporal} with $r=C\Bigl\{Bp+(1+B)\bigl(\sqrt{p\log\log(8n)}+\log\log(8n)\bigr)\Bigr\}$ satisfies
	$$\sup_{\substack{\theta\in\Theta_0(p,n)\\\Sigma\in\mathcal{C}(p,n,B)}}\mathbb{E}_{\theta,\Sigma}\psi_{\mathrm{Temp}} + \sup_{\substack{\theta\in\Theta(p,n,p,\rho)\\\Sigma\in\mathcal{C}(p,n,B)}}\mathbb{E}_{\theta,\Sigma}(1-\psi_{\mathrm{Temp}})\leq\epsilon,$$
	as long as $\rho^2\geq 32C\Bigl\{Bp+(1+B)\bigl(\sqrt{p\log\log(8n)}+\log\log(8n)\bigr)\Bigr\}$.
\end{thm}
Our final result provides the complementary lower bound.
\begin{thm}\label{thm:lower-temporal}
	Assume that $B \leq D\sqrt{n/p}$ for some $D > 0$, and let
	\begin{equation}
	\label{eq:minimax-B}
	\rho_{\mathrm{Temp}}^{*2} \equiv \rho_{\mathrm{Temp}}^*(p,n,p,B)^2 := Bp + (1+B)\bigl\{\sqrt{p\log\log(8n)}\vee \log\log(8n)\bigr\}.
	\end{equation}
	Then given $\epsilon > 0$, there exist $c_{\epsilon,D} > 0$, depending only on $\epsilon$ and $D$, and $p_\epsilon \in \mathbb{N}$, depending only on $\epsilon$, such that $\mathcal{R}(c\rho_{\mathrm{Temp}}^*) \geq 1-\epsilon$ whenever $c \in (0,c_{\epsilon,D})$ and $p \geq p_\epsilon$. 
\end{thm}
Together, Theorems~\ref{thm:upper-temporal} and~\ref{thm:lower-temporal} reveal the rate of the minimax detection boundary when $B \lesssim \sqrt{n/p}$. 
Observe that when $B=0$, the rate (\ref{eq:minimax-B}) becomes $\sqrt{p\log\log(8n)}\vee\log\log(8n)$, which matches (\ref{eq:minimax-rate}) when $s=p$. When $B>0$, the rate (\ref{eq:minimax-B}) has an extra multiplicative factor $1+B$ and an extra additive factor $Bp$, which are present for different reasons. Due to the dependence of the time series, one can think of $n/(1+B)$ and $\rho^2/(1+B)$ as being the effective sample size and signal strength respectively, instead of $n$ and $\rho^2$ for the independent case, and this leads to the presence of the multiplicative factor $1+B$. On the other hand, the additive term $Bp$ arises from the fact that $\mathbb{E}_{\theta,\Sigma}\|Y_t\|^2-p$ under the null hypothesis is not known completely due to the unknown covariance structure $\Sigma\in\mathcal{C}(p,n,B)$.  In fact, in the construction of the lower bound, the relevant zero mean Gaussian distribution with unknown covariance can be approximated by a location mixture of Gaussians with known identity covariance.  This allows us to relate the difficulties of the two problems.  When $B=0$, the class $\mathcal{C}(p,n,B)$ becomes a singleton, and we know that $\mathbb{E}_{\theta,\Sigma}\|Y_t\|^2=p$ under the null, so this additional term disappears.

\section{Proofs}\label{sec:proofs}

\subsection{Proofs of results in Section \ref{sec:main}}


\begin{proof}[Proof of Proposition~\ref{thm:upper}]
	In this proof, we seek to control Type I and Type II errors using tail probability bounds for chi-squared random variables and a version with truncated summands; these are given as Lemmas~\ref{lem:chi-square-tail} and~\ref{lem:trunc-chi-square-tail} respectively.  Fixing $\epsilon \in (0,1)$, set $C=C(\epsilon) := 50C_1/\epsilon$, where the universal constant $C_1 \geq 1$ is taken from Lemma~\ref{lem:trunc-chi-square-var}.  We first consider the case where $s\geq\sqrt{p\log\log(8n)}$.  Then $a=0$, so that $A_{t,a}=\sum_{j=1}^pY_t(j)^2-p$. Therefore, for any $\theta\in\Theta_0(p,n)$, we have $A_{t,a}\sim \chi_p^2-p$.  Then, by a union bound and Lemma~\ref{lem:chi-square-tail}, we obtain that with $x := \frac{C}{9}\log\log(8n)$,
	\begin{align}
\nonumber	\mathbb{E}_\theta\psi = \mathbb{P}_\theta\biggl(\max_{t\in\mathcal{T}}A_{t,0}>C\sqrt{p \log \log (8n)}\biggr)&\leq \mathbb{P}_{\theta}\Bigl(\max_{t\in\mathcal{T}}A_{t,0}>2\sqrt{px} + 2x\Bigr) \\
\label{eq:exp-to-kill-p-1}	&\leq 2\log(en) e^{-x} \leq \frac{\epsilon}{2},
	\end{align}
	where the final inequality holds because $C \geq 9+9\log(4/\epsilon)$.
	
	Now suppose that $\theta\in\Theta(p,n,s,\rho)$. For any $\theta\in\Theta(p,n,s,\rho)$, there exists some $t_0\in [n-1]$ such that $X_1,\ldots,X_{t_0}\stackrel{\mathrm{iid}}{\sim}N_p(\mu_1,I_p)$ and $X_{t_0+1},\ldots,X_n\stackrel{\mathrm{iid}}{\sim}N_p(\mu_2,I_p)$, where the vectors $\mu_1$ and $\mu_2$ satisfy $\frac{t_0(n-t_0)}{n}\|\mu_1-\mu_2\|^2\geq\rho^2$. Without loss of generality, we may assume that $t_0\leq n/2$, since the case $t_0>n/2$ can be handled by a symmetric argument. By the definition of $\mathcal{T}$, there exists a unique $\wt{t}\in\mathcal{T}$ such that $t_0/2 < \wt{t}\leq t_0$.  Now $A_{\wt{t},a}\sim \chi_{p,\delta^2}^2-p$, where the non-centrality parameter $\delta^2$ satisfies
	$$\delta^2=\frac{\wt{t}\|\mu_1-\mu_2\|^2}{2}\geq\frac{t_0\|\mu_1-\mu_2\|^2}{4}\geq \frac{t_0(n-t_0)}{4n}\|\mu_1-\mu_2\|^2\geq\frac{\rho^2}{4}.$$
	Therefore, by Chebychev's inequality,
	\begin{align}
	\label{eq:usedlater-1}
	\mathbb{E}_{\theta}(1-\psi) &\leq \mathbb{P}_{\theta}\biggl(\max_{t\in\mathcal{T}} \|Y_t\|^2 - p \leq \frac{\rho^2}{32}\biggr) \leq \mathbb{P}_{\theta}\biggl(\|Y_{\wt{t}}\|^2 \! - \! p \leq \frac{\delta^2}{8}\biggr) \leq \frac{2(p+2\delta^2)}{(7/8)^2\delta^4} \nonumber \\
	&\leq \frac{2(p+\rho^2/2)}{(7/32)^2\rho^4}
	\leq \frac{2}{49C^2\log \log (8n)} + \frac{32}{49C\sqrt{p\log \log (8n)}} \leq \frac{\epsilon}{2},
	\end{align}
	since $C \geq 49/(68\epsilon)$.
	
	We now consider the case where $s<\sqrt{p\log\log(8n)}$, and first suppose that $\theta\in\Theta_0(p,n)$. By Lemma \ref{lem:trunc-chi-square-tail} and a union bound, we have
	\begin{align}
	\mathbb{E}_\theta\psi = \mathbb{P}_\theta\biggl(\max_{t\in\mathcal{T}}A_{t,a}>Cr^*\biggr) &\leq \mathbb{P}_{\theta}\biggl(\max_{t\in\mathcal{T}}A_{t,a}> \sqrt{pe^{-a^2/2}x}+x\biggr) \nonumber \\
	&\leq 2\log(en)e^{-x} \leq \frac{\epsilon}{2}, \label{eq:type-1-sparse-main}
	\end{align}
	where we still take $x = \frac{C}{9}\log\log(8n)$.
	
	Finally, for $\theta\in\Theta(p,n,s,\rho)$, we define $\wt{t},\mu_1,\mu_2$ as in the dense case. 
	Now
	$$\max_{t\in\mathcal{T}}A_{t,a}\geq A_{\wt{t},a}=\sum_{j=1}^p \bigl(Y_{\wt{t}}(j)^2-\nu_a\bigr)\mathbbm{1}_{\{|Y_{\wt{t}}(j)|\geq a\}},$$
	where $Y_{\wt{t}}(j)\sim N(\Delta_j,1)$, with $\Delta_j:=\sqrt{\frac{\wt{t}}{2}}\bigl\{\mu_1(j)-\mu_2(j)\bigr\}$.  By Lemma \ref{lem:trunc-chi-square-mean}, 
	\begin{align*}
	\mathbb{E}A_{\wt{t},a} &\geq \frac{1}{2}\sum_{j:|\Delta_j|\geq 8a} \Delta_j^2 = \frac{1}{2}\Biggl(\sum_{j=1}^p\Delta_j^2 - \sum_{j:0 < |\Delta_j|<8a} \! \Delta_j^2\Biggr) \geq \frac{1}{2}\bigl(\delta^2 \!-\! 64sa^2\bigr) \geq \frac{\delta^2}{4},
	\end{align*}
	where the last inequality uses the fact that $4\delta^2 \geq \rho^2\geq 8Csa^2$. Moreover, by Lemma~\ref{lem:trunc-chi-square-var}, we have
	\[
		\Var(A_{\wt{t},a}) = \sum_{j=1}^p\Var\bigl\{(Y_{\wt{t}}(j)^2-\nu_a)\mathbbm{1}_{\{|Y_{\wt{t}}(j)|\geq a\}}\bigr\} \leq C_1\bigl(pe^{-a^2/4} + sa^4 + \delta^2\bigr).
	\]
	By Chebychev's inequality, we deduce that
	\begin{align}
	\nonumber\mathbb{E}_{\theta}(1-\psi) &= \mathbb{P}_{\theta}\left(\max_{t\in\mathcal{T}}A_{t,a}\leq \frac{\rho^2}{32}\right)
	 \leq \mathbb{P}_{\theta}\left(A_{\wt{t},a}\leq \frac{\delta^2}{8}\right) \\
	 &\leq \frac{\Var(A_{\wt{t},a})}{\bigl(\mathbb{E}A_{\wt{t},a}- \delta^2/8\bigr)^2} \leq \frac{C_1\bigl(pe^{-a^2/4} + sa^4 + \delta^2\bigr)}{\delta^4/2^6}\nonumber \\
	\label{eq:usedlater-2} &\leq \frac{C_1pe^{-a^2/4} + C_1sa^4 + C_1\rho^2/4}{\rho^4/2^{10}} \leq \frac{C_1}{C^2} + \frac{16C_1}{C^2} + \frac{8C_1}{C} \leq \frac{\epsilon}{2},
	\end{align}
	as required. 
\end{proof}

The proof of Proposition~\ref{thm:lower} below is based on the lower bound technique that involves bounding the chi-squared divergence.  
\begin{proof}[Proof of Proposition~\ref{thm:lower}]
We only need to derive the lower bound for $n$ that is sufficiently large. This is because when $n$ is bounded, the minimax rate is reduced to the formula~\eqref{eq:collier}, and the derivation of the lower bound follows the same argument as in \cite{collier2017minimax}.  The strategy for our lower bound is to construct a suitable prior distribution on the alternative hypothesis parameter space and to bound the total variation distance between the null distribution and the mixture distribution induced by the prior.  More precisely, by Lemmas~\ref{lem:lower} and~\ref{lem:chi-square-cal}, given $\eta > 0$, it suffices to find a probability measure~$\nu$ with $\supp(\nu)\subseteq\Theta(p,n,s,\rho)$ and a universal constant $c > 0$ such that
	\begin{equation}
	\mathbb{E}_{(\theta_1,\theta_2)\sim \nu\otimes\nu}\exp(\iprod{\theta_1}{\theta_2})\leq 1 + \eta,\label{eq:lower-goal}
	\end{equation}
	whenever $\rho=c\rho^*$.
	
	We first consider the case when $s\geq\sqrt{p\log\log(8n)}$. We define $\nu$ to be the distribution of $\theta = (\theta_{j\ell}) \in \Theta(p,n,s,\rho)$ with $\rho := \sqrt{s}\beta/\sqrt{2}$ for some $\beta = \beta(p,n,s)$ to be defined later, generated according to the following sampling process:
	\begin{enumerate}
		\item Uniformly sample a subset $S \subseteq [p]$ of cardinality $s$;
		\item Independently of $S$, generate $k \sim \mathrm{Unif}\{0,1,2,\ldots,\floor{\log_2(n/2)}\}$;
		\item Independently of $(S,k)$, sample $u = (u_1,\ldots,u_p) \in\mathbb{R}^p$, where \\$u_1,\ldots,u_p \stackrel{\mathrm{iid}}{\sim}\text{Unif}(\{-1,1\})$;
		\item Given the triplet $(S,k,u)$ sampled in the previous steps, define $\theta_{j\ell} :=\frac{\beta}{\sqrt{2^k}}u_j$ for all $(j,\ell)\in S\times [2^k]$ and $\theta_{j\ell}:=0$ otherwise.
	\end{enumerate}
	Since
    $$\frac{2^k(n-2^k)}{n}\frac{\beta^2}{2^k}\sum_{j\in S}u_j^2=s\beta^2\frac{n-2^k}{n}\geq\frac{s\beta^2}{2},$$
    we have $\supp(\nu)\subseteq\Theta(p,n,s,\rho)$ with $\rho^2=s\beta^2/2$.
	Suppose we independently sample triplets $(S,k,u)$ and $(T,l,v)$ from the first three steps and use these two triplets to construct $\theta_1$ and $\theta_2$ according to the fourth step. Then
	$$\iprod{\theta_1}{\theta_2}=\bigl(2^k\wedge 2^l\bigr)\frac{\beta^2}{\sqrt{2^{k+l}}}\sum_{j\in S\cap T}u_jv_j=\frac{\beta^2}{2^{|l-k|/2}}\sum_{j\in S\cap T}u_jv_j.$$
	Thus
	$$\mathbb{E}_{(\theta_1,\theta_2)\sim \nu\otimes\nu}\exp(\iprod{\theta_1}{\theta_2})=\mathbb{E}\exp\biggl(\frac{\beta^2}{2^{|l-k|/2}}\sum_{j\in S\cap T}u_jv_j\biggr),$$
	where the expectation is over the joint distribution of $(S,k,u,T,l,v)$.  But we also have that $u_jv_j\stackrel{\mathrm{iid}}{\sim}\mathrm{Unif}(\{-1,1\})$, so
	\begin{align*}
	\mathbb{E}_{(\theta_1,\theta_2)\sim \nu\otimes\nu}\exp(\iprod{\theta_1}{\theta_2}) &= \mathbb{E}\biggl\{\mathbb{E}\biggl(\frac{1}{2}e^{\beta^2/2^{|l-k|/2}} +\frac{1}{2}e^{-\beta^2/2^{|l-k|/2}}\biggr)\biggr\}^{|S\cap T|} \\
	&\leq \mathbb{E}\exp\left(|S\cap T|\frac{\beta^4}{2^{|l-k|+1}}\right),
	\end{align*}
	where the final inequality uses the fact that $(e^x+e^{-x})/2\leq e^{x^2/2}$ for $x \in \mathbb{R}$ and Jensen's inequality. Note that $|S\cap T|$ is distributed according to the hypergeometric distribution\footnote{The $\mathrm{Hyp}(p,s,r)$ distribution models the number of white balls drawn when sampling $r$ balls without replacement from an urn containing $p$ balls, $s$ of which are white.} $\mathrm{Hyp}(p,s,s)$.  By the fact that the $\mathrm{Hyp}(p,s,s)$ distribution is no larger, in the convex ordering sense, that the binomial distribution $\mathrm{Bin}(s,s/p)$ \citep[][Theorem~4]{hoeffding1963}, we have
	\begin{align}
	\nonumber \mathbb{E}\exp\left(|S\cap T|\frac{\beta^4}{2^{|l-k|+1}}\right) &\leq \biggl\{\mathbb{E}\biggl(1-\frac{s}{p}+\frac{s}{p}e^{\beta^4/2^{|l-k|+1}}\biggr)\biggr\}^s \\
	\label{eq:used-in-asymp-sparse} &\leq \mathbb{E}\biggl\{\biggl(1+\frac{s}{2p}\frac{\beta^4}{2^{|l-k|}}e^{\beta^4/2^{|l-k|+1}}\biggr)^s\biggr\} =: \mathbb{E}L(l,k),
	\end{align}
	say, where we have used $e^x-1\leq xe^x$ for all $x\geq 0$ and Jensen's inequality to derive the last inequality above.  
	From now on, we set $\beta := \bigl\{c_1ps^{-2}\log\log(8n)\bigr\}^{1/4}$, where $c_1 = c_1(\eta) \in (0,1/4]$ will be chosen to be sufficiently small. The condition $s\geq\sqrt{p\log\log(8n)}$ ensures that $\beta \leq 1$.  We first claim that 
	\begin{equation}
	\label{Eq:FirstTerm}
	\mathbb{E}\bigl\{L(l,k)\mathbbm{1}_{\{l=k\}}\bigr\} 
	\leq \Bigl\{\Bigl(1+\frac{\eta}{4}\Bigr)\mathbb{P}(l=k)\Bigr\} \vee \frac{\eta}{4},
	\end{equation}
	provided that $c_1\leq \eta\log\left(1+\frac{\eta}{4}\right)/8$.  To see this, first note that for $n\geq \exp(\exp(8/\eta))/8$, we have
	\begin{align*}
	\mathbb{E}\bigl\{L(l,k)\mathbbm{1}_{\{l=k\}}\bigr\} &\leq \biggl(1+\frac{c_1}{s}\log\log(8n)\biggr)^s\mathbb{P}(l=k) \leq \frac{\log^{1/4}(8n)}{1+\lfloor\log_2(n/2)\rfloor} \\
	&\leq \frac{\eta}{8} \log\log(8n)\frac{\log^{1/4}(8n)}{1+\lfloor\log_2(n/2)\rfloor} \leq \frac{\eta}{4}.
	\end{align*}
	On the other hand, when $n<\exp(\exp(8/\eta))/8$, we have
	\[
	\mathbb{E}\bigl\{L(l,k)\mathbbm{1}_{\{l=k\}}\bigr\} \leq \log^{c_1}(8n) \mathbb{P}(l=k) \leq e^{8c_1/\eta}\mathbb{P}(l=k) \leq \Bigl(1+\frac{\eta}{4}\Bigr)\mathbb{P}(l=k).
	\]
	Moreover,
	\begin{align}
	\mathbb{E}\bigl\{L(l,k)&\mathbbm{1}_{\{0 < |l-k|\leq (\eta/8)\log\log(8n)\}}\bigr\} \nonumber\\
	&\leq\biggl(1+\frac{c_1}{s}\log\log(8n)\biggr)^s\mathbb{P}\biggl(0 < |l-k|\leq \frac{\eta}{8}\log\log(8n)\biggr) \nonumber \\
	\label{eq:Changed1}&\leq\log^{1/4}(8n)\frac{\eta\log\log(8n)}{4\bigl(1+\lfloor\log_2(n/2)\rfloor\bigr)} \leq \frac{\eta}{2}.
	\end{align}
	For the third term, we write $a_\eta := \sup_{n\geq 2}\frac{\log\log(8n)}{\log^{(\eta/8) \log(2)}(8n)}$.  By reducing $\eta > 0$ and $c_1 = c_1(\eta)$ if necessary, we may assume that $c_1a_\eta \leq \eta/8 \leq 1/2$, so that
	\begin{align}
	\label{Eq:Changed} 
	\mathbb{E}\bigl\{L(l,k)&\mathbbm{1}_{\{|l-k|> (\eta/8)\log\log(8n)\}}\bigr\}\nonumber\\
	 &\leq \biggl(1+\frac{c_1a_\eta}{s}\biggr)^s\mathbb{P}\bigl\{|l-k|> (\eta/8)\log\log(8n)\bigr\} \nonumber \\
	&\leq (1 + 2c_1a_\eta)\mathbb{P}\bigl\{|l-k|> (\eta/8)\log\log(8n)\bigr\} \nonumber \\
	&\leq \biggl(1 + \frac{\eta}{4}\biggr)\mathbb{P}\bigl\{|l-k|> (\eta/8)\log\log(8n)\bigr\}. 
	\end{align}
	From~\eqref{Eq:FirstTerm},~\eqref{eq:Changed1} and~\eqref{Eq:Changed}, we conclude that
	\[
	\mathbb{E}\bigl\{L(l,k)\bigr\} \leq 1+\eta,
	\]
	which establishes~\eqref{eq:lower-goal} in the case $s \geq \sqrt{p \log \log(8n)}$.
	
	We now consider the case $s<\sqrt{p\log\log (8n)}$ and $s\log\left(\frac{ep\log\log(8n)}{s^2}\right)\geq \log\log(8n)$. The goal is to derive a lower bound with rate $s\log\left(\frac{ep\log\log(8n)}{s^2}\right)$. We use the same $\nu$ specified in the previous case except that in the third step, we set $u_j=1$ for all $j\in S$. With this modification of~$\nu$, we have
	$\iprod{\theta_1}{\theta_2}=|S\cap T| \frac{\beta^2}{2^{|l-k|/2}}$. Again, $|S\cap T|$ is distributed according to the hypergeometric distribution $\mathrm{Hyp}(p,s,s)$, and 
	\begin{align}
	\label{eq:mgfb3step} &\mathbb{E}_{(\theta_1,\theta_2)\sim \nu\otimes\nu}\exp(\iprod{\theta_1}{\theta_2}) = \mathbb{E}\exp\left(|S\cap T|\frac{\beta^2}{2^{|l-k|/2}}\right) \\
	\label{eq:no-where} \leq \biggl\{&\mathbb{E}\left(1-\frac{s}{p}+\frac{s}{p}e^{\beta^2/2^{|l-k|/2}}\right)\biggr\}^s \leq \mathbb{E}\biggl\{\left(1+\frac{s}{p}e^{\beta^2/2^{|l-k|/2}}\right)^s\biggr\} =: \mathbb{E}R(l,k),
	\end{align}
	say.  We take $\beta :=\log^{1/2}\left(\frac{c_2p\log\log(8n)}{s^2}\right)$, where $c_2 = c_2(\eta) \in (0,1/4]$ will be chosen sufficiently small. 
	Parallel to the bounds for $\mathbb{E}L(l,k)$, we will split into three terms. For the first term, we have
	\begin{align*}
	\mathbb{E}\bigl\{R(l,k)\mathbbm{1}_{\{l=k\}}\bigr\} &\leq \biggl(1+\frac{c_2}{s}\log\log(8n)\biggr)^s\mathbb{P}(l=k) \\
    &\leq \Bigl\{\Bigl(1+\frac{\eta}{4}\Bigr)\mathbb{P}(l=k)\Bigr\} \vee \frac{\eta}{4},
	\end{align*}
	as before, as long as $c_2\leq \eta\log\left(1+\frac{\eta}{4}\right)/8$.  For the second term, 
	\begin{align*}
	\mathbb{E}\bigl\{R(l,k)&\mathbbm{1}_{\{0<|l-k|\leq (\eta/8)\log\log(8n)\}}\bigr\}\\ 
	&\leq \biggl(1+\frac{s}{p}e^{\beta^2}\biggr)^s\mathbb{P}\left(0<|l-k|\leq\frac{\eta}{8}\log\log(8n)\right) \\
	&\leq \left(1+\frac{\log\log(8n)}{4s}\right)^s\frac{\eta\log\log(8n)}{4\bigl(1+\lfloor\log_2(n/2)\rfloor\bigr)} \leq \frac{\eta}{2}.
	\end{align*}
	For the third term, define $b_{\eta}:=\sup_{n\geq 2} \exp\biggl(\frac{\log \log\log(8n)}{\log^{(\eta/16)\log 2}(8n)}\biggr)$. By reducing $c_2 = c_2(\eta)$ if necessary, we may assume that $c_2\leq \log(1+\eta/4)/b_{\eta}$.  Then 
	\begin{align*}
	\mathbb{E}\bigl\{&R(l,k)\mathbbm{1}_{\{|l-k|> (\eta/8)\log\log(8n)\}}\bigr\} \\
	&\leq \Biggl\{1+\frac{s}{p}\exp\biggl(\frac{\log(c_2p/s^2)+\log\log\log(8n)}{\log^{(\eta/16)\log 2}(8n)}\biggr)\Biggr\}^s\\ 
	&\hspace{6cm} \times\mathbb{P}\bigl\{|l-k|> (\eta/8)\log\log(8n)\bigr\}\\
	&\leq e^{c_2b_{\eta}}\mathbb{P}\bigl\{|l-k|> (\eta/8)\log\log(8n)\bigr\}\\
	&\leq \Bigl(1+\frac{\eta}{4}\Bigr)\mathbb{P}\bigl\{|l-k|> (\eta/8)\log\log(8n)\bigr\},
	\end{align*}
	which establishes~\eqref{eq:lower-goal} when $s<\sqrt{p\log\log (8n)}$ and $s\log\left(\frac{ep\log\log(8n)}{s^2}\right)\geq \log\log(8n)$.
	
	The final case is $s<\sqrt{p\log\log (8n)}$ and $s\log\left(\frac{ep\log\log(8n)}{s^2}\right)< \log\log(8n)$. Notice that in our definition of the parameter space $\Theta^{(t_0)}(p,n,s,\rho)$, if we restrict $\mu_1$ and $\mu_2$ to agree in all coordinates except perhaps the first, then the testing problem is equivalent to testing between $\Theta_0(1,n)$ and $\Theta(1,n,1,\rho)$. Therefore, the lower bound construction in \cite{gao2017minimax} applies directly here and we obtain the rate $\log\log(8n)$.
	
	The result follows.
\end{proof}

The proof of Theorem~\ref{thm:adaptive} uses several arguments from the proof of Proposition~\ref{thm:upper}.
\begin{proof}[Proof of Theorem \ref{thm:adaptive}]
We first bound $\mathbb{E}\psi_{\rm adaptive}$ for any $\theta\in\Theta_0(p,n)$, and let $\epsilon \in (0,1)$. By a union bound, we have
\begin{equation}
\mathbb{E}_{\theta}\psi_{\rm adaptive} \leq \sum_{s\in\mathcal{S}: s<\sqrt{p\log\log(8n)}}\mathbb{E}_{\theta}\psi^{(s)} + \mathbb{E}_{\theta}\psi^{(p)}. \label{eq:smart-union}
\end{equation}
By the same argument as that used in the proof of Proposition~\ref{thm:upper}, we have $\mathbb{E}_{\theta}\psi^{(p)}\leq \epsilon/4$ as long as $C = C(\epsilon) > 0$ is chosen sufficiently large (in particular, it will need to be at least as large as the choice of $C$ in the proof of Proposition~\ref{thm:upper}).

For $s<\sqrt{p\log\log(8n)}$, we recall that $a^2 = a^2(s) =4\log\left(\frac{ep\log\log(8n)}{s^2}\right)$.  As in (\ref{eq:type-1-sparse-main}), we have
$$\mathbb{E}_{\theta}\psi^{(s)}\leq 2\log(en)e^{-x},$$
for any $x$ such that  
$$\sqrt{\frac{s^4}{e^2p\log^2\log(8n)}x} + x \leq Cr^*(p,n,s). $$
The choice
\[x=\frac{C}{2}\left(\frac{p\log^2 \log(8n)}{s^2}\wedge r^*(p,n,s)\right)\]
satisfies this condition provided that $C\geq 2$. 
This choice of $x$ gives the tail bound
$$\mathbb{E}_{\theta}\psi^{(s)}\leq 2\log(en)\exp\left(-\frac{C}{2}\frac{p\log^2\log(8n)}{s^2}\right) + 2\log(en)\exp\left(-\frac{C}{2}r^*(p,n,s)\right).$$
Moreover, provided we choose $C = C(\epsilon) > 0$ sufficiently large, we have
\begin{align*}
2\log(en)&\sum_{s\in\mathcal{S}: s<\sqrt{p\log\log(8n)}}\exp\left(-\frac{C}{2}\frac{p\log^2\log(8n)}{s^2}\right) \\
&\leq 2\log(en)\sum_{k=0}^{\infty}\exp\left(-\frac{C}{2}4^k\log\log(8n)\right) \leq \frac{\epsilon}{8}.
\end{align*}
Similarly,
\begin{align*}
& 2\log(en)\sum_{s\in\mathcal{S}: s<\sqrt{p\log\log(8n)}}\exp\left(-\frac{C}{2}r^*(p,n,s)\right) \\
&\hspace{0.5cm}\leq 2\log(en)\sum_{s\in\mathcal{S}: s<\sqrt{p\log\log(8n)}}\exp\biggl\{-\frac{C}{4}s\log\left(\frac{ep\log\log(8n)}{s^2}\right) - \frac{C}{4}\log\log(8n)\biggr\} \\
&\hspace{0.5cm}\leq \frac{\epsilon}{8}\sum_{s\in[p]: s<\sqrt{p\log\log(8n)}}\exp\biggl\{-\frac{C}{4}s\log\left(\frac{ep\log\log(8n)}{s^2}\right) \biggr\} \leq \frac{\epsilon}{8}.
\end{align*}
Therefore we have $\mathbb{E}_{\theta}\psi_{\rm adaptive} \leq \epsilon/2$ by (\ref{eq:smart-union}).

Finally, for $\theta\in\Theta(p,n,s,\rho)$, we bound $\mathbb{E}_{\theta}(1-\psi_{\rm adaptive})$. By the definition of $\mathcal{S}$, there exists a unique $\wt{s}\in\mathcal{S}$ such that $s/2<\wt{s}\leq s$.  Moreover, $\Theta(p,n,s,\rho)\subseteq\Theta(p,n,p\wedge2\wt{s},\rho)$, so by the same argument used in the proof of Proposition~\ref{thm:upper}, we have
$$\mathbb{E}_{\theta}(1-\psi_{\rm adaptive})\leq \mathbb{E}_{\theta}(1-\psi^{(\wt{s})})\leq \frac{\epsilon}{2},$$
as long as $\rho^2\geq 32Cr^*(p,n,p\wedge2\wt{s})$. But $\wt{s}\leq s$, so $r^*(p,n,p\wedge 2\wt{s})\leq 2r^*(p,n,s)$, which implies that $\rho^2\geq 64Cr^*(p,n,s)$ is a sufficient condition for controlling the error under the alternative. 
\end{proof}

For Theorem \ref{thm:asymp-dense} and Theorem \ref{thm:asymp-sparse}, we will prove lower and upper bounds separately.

\begin{proof}[Proof of Theorem \ref{thm:asymp-dense} (lower bound)]
Since this proof is asymptotic, we assume in many places and without further comment (both here and in the upper bound proof that follows) that $n$ is sufficiently large in developing our bounds. Let $f_{0n} := \phi$ be the density function of the standard normal distribution on $\mathbb{R}^{p\times n}$. Define $f_{1n}(x) := \int_{\supp(\nu_n)} \phi(x-\theta) \, d\nu_n(\theta)$, where $\nu_n$ is the distribution of $\theta$ when $\theta$ is generated according to the following sampling process:
\begin{enumerate}
    	\item Uniformly sample a subset $S \subseteq [p]$ of cardinality $s$;
    	\item Independently of $S$, generate $k \sim \mathrm{Unif}(\mathcal{G}_n)$ where 
\begin{align*}
\mathcal{G}_n := \bigl\{0,&\floor{c\log\log\log n},2\floor{c\log\log\log n},\ldots, \\
&\floor{c\log\log\log n}\floor{\log_2(\sqrt{n})/\floor{c\log\log\log n}}\bigr\},
\end{align*}
    	for some constant $c > 0$ to be chosen later;
    	\item Independently of $(S,k)$, sample $u = (u_1,\ldots,u_p) \in\mathbb{R}^p$, where \\$u_1,\ldots,u_p \stackrel{\mathrm{iid}}{\sim}\text{Unif}(\{-1,1\})$;
    	\item Given the triplet $(S,k,u)$ sampled in the previous steps, define $\theta_{j\ell} :=\frac{\beta}{\sqrt{2^k}}u_j$ for all $(j,\ell)\in S\times [2^k]$ and $\theta_{j\ell}:=0$ otherwise, where $\beta^2 = (2-\epsilon)\sqrt{\frac{p\log\log n}{s^2}}$ for some small, constant $\epsilon>0$.
    \end{enumerate}
Since
\[\frac{2^k(n-2^k)}{n}\cdot \frac{\beta^2}{2^k} \sum_{j\in S} u_j^2 \geq \frac{n-\sqrt{n}}{n}(2-\epsilon)\sqrt{p\log\log n},\]
we have $\supp(\nu_n)\subseteq\Theta(p,n,s,\rho)$ with $\rho^2=(2-2\epsilon)\sqrt{p\log\log n}$, which, following the argument in the proof of Lemma~\ref{lem:lower}, yields that 
$$\mathcal{R}(\rho)\geq \int_{\mathbb{R}^{p \times n}} f_{0n}\wedge f_{1n}.$$
Therefore, in order to show that $\mathcal{R}(\rho)\rightarrow 1$, it suffices to establish that $\frac{f_{1n}}{f_{0n}}\rightarrow 1$ in $f_{0n}$-probability as $n\rightarrow\infty$. By a truncated second moment argument, given as Lemma~\ref{Lemma:TruncatedSM}, it suffices to choose some measurable set $A_n\subseteq\mathbb{R}^{p\times n}$ and establish the following two conditions:
	\begin{enumerate}
		\item $\mathbb{E}_{X\sim f_{0n}}\bigl\{\frac{f_{1n}(X)}{f_{0n}(X)}\indc{X\in A_n}\bigr\}\rightarrow 1$;
		\item $\mathbb{E}_{X\sim f_{0n}}\bigl\{\bigl(\frac{f_{1n}(X)}{f_{0n}(X)}\bigr)^2\indc{X\in A_n}\bigr\}\rightarrow 1$.
	\end{enumerate}
By definition of $f_{1n}$ and Lemma \ref{lem:chi-square-cal}, the above two conditions are equivalent to:
	\begin{enumerate}
		\item $\mathbb{P}_{X\sim f_{1n}}(X\in A_n)\rightarrow 1$;
		\item $\mathbb{E}_{(\theta_1,\theta_2)\sim\nu_n\otimes\nu_n}\bigl\{\mathbb{P}_{X\sim N_{p\times n}(\theta_1+\theta_2,I)}(X\in A_n)\exp\bigl(\iprod{\theta_1}{\theta_2}\bigr)\bigr\}\rightarrow 1$.
	\end{enumerate}
    Fix $\epsilon_1>0$ and define $\mathcal{I}_n:= \{2^i:i \in \mathcal{G}_n\}$. Then we choose the truncation set to be 
    \[A_n:=\left\{X\in \mathbb{R}^{p\times n}:\max_{t\in \mathcal{I}_n}\norm{X_1+\dots+X_t}_2^2/t\leq p+(2+\epsilon_1)\sqrt{p\log\log n}\right\}.\]
    
\noindent\textit{Proof of Condition 1.}    
    To prove that $\mathbb{P}_{X\sim f_{1n}}(X\in A_n)\rightarrow 1$, it suffices to show that $\sup_{\theta \in \supp(\nu_n)}\mathbb{P}_{\theta}(X\notin A_n)\rightarrow 0$. 
    Assume that the true changepoint location in $\theta$ is $t_0$. Then, $\norm{X_1+\dots+X_t}_2^2/t$ follows a non-central chi-squared distribution with degrees of freedom $p$ and non-centrality parameter 
    \[\frac{(2-\epsilon)\sqrt{p\log\log n}}{(t/t_0)\vee (t_0/t)}.\] 
    We therefore divide the time grid into two parts $\mathcal{I}_n=\mathcal{I}^1_n\cup \mathcal{I}^2_n$, where $\mathcal{I}^1_n:=\{t_0\}$ and $\mathcal{I}^2_n:=\mathcal{I}_n\setminus \{t_0\}$. Since $\floor{c\log\log\log n}\rightarrow \infty$, the non-centrality parameter for $t\in \mathcal{I}_n^2$ is smaller than $\frac{\epsilon_1}{2}\sqrt{p\log\log n}$. Then 
    \begin{align*}
    \mathbb{P}_{\theta}(X\notin A_n) &\leq  \sum_{t\in \mathcal{I}_n^1\cup \mathcal{I}_n^2}\mathbb{P}_\theta\biggl(\frac{\norm{X_1+\dots+X_t}_2^2}{t} > p+(2+\epsilon_1)\sqrt{p\log\log n}\biggr)\\
    &\leq \mathbb{P}\left(\chi^2_{p,(2-\epsilon)\sqrt{p\log\log n}} > p+(2+\epsilon_1)\sqrt{p\log\log n}\right)\\
    &\hspace{1cm}+(\log n)\mathbb{P}\left(\chi^2_{p,\epsilon_1\sqrt{p\log\log n}/2} > p+(2+\epsilon_1)\sqrt{p\log\log n}\right)\\
    &\rightarrow 0,
    \end{align*}
    where the last line is by by Lemma~\ref{lem:noncentral-chi-square-tail}, and we have used the fact that $\frac{\log\log n}{\sqrt{p\log\log n}}\rightarrow 0$ as $n\rightarrow \infty$ by assumption. Moreover, the bound we just obtained is uniform over all $\theta\in\supp(\nu_n)$, so $\mathbb{P}_{X\sim f_{1n}}(X\in A_n)\rightarrow 1$.
    
\vspace{.1in}
    
\noindent\textit{Proof of Condition 2.}   
    We independently sample $(S,k,u)$ and $(T,l,v)$ with distribution $\nu_n$, and define $\theta_1$ using $(S,k,u)$ and $\theta_2$ using $(T,l,v)$. Then
    \begin{align*}
      &\mathbb{E}_{(\theta_1,\theta_2)\sim\nu_n\otimes\nu_n}\bigl\{\mathbb{P}_{X\sim N_{p\times n}(\theta_1+\theta_2,I)}(X\in A_n)\exp\bigl(\iprod{\theta_1}{\theta_2}\bigr)\bigr\} \\
      &\hspace{1.5cm}\leq\mathbb{E}_{k,l}\bigl[\mathbb{E}_{S,u,T,v}\bigl\{\mathbb{P}_{X\sim N_{p\times n}(\theta_1+\theta_2,I)}(X\in A_n)\exp\bigl(\iprod{\theta_1}{\theta_2}\bigr)\bigr\}\indc{k=l}\bigr] \\
    &\hspace{5.5cm}+\mathbb{E}_{k,l}\bigl[\mathbb{E}_{S,u,T,v}\bigl\{\exp\left(\iprod{\theta_1}{\theta_2}\right)\bigr\}\indc{k\neq l}\bigr].
    \end{align*}
    By the same calculation as in the proof of Proposition~\ref{thm:lower}, the second term above can be bounded as follows:
    \begin{align*}
      \mathbb{E}_{k,l}\bigl\{\mathbb{E}_{S,u,T,v}\bigl(e^{\iprod{\theta_1}{\theta_2}}\bigr)\indc{k\neq l}\bigr\} &\leq \mathbb{E}_{k,l}\biggl\{\biggl(1+\frac{s}{2p}\frac{\beta^4}{2^{|l-k|}}e^{\beta^4/2^{|l-k|+1}}\biggr)^s\indc{k\neq l}\biggr\}\\
    &\leq \mathbb{E}_{k,l}\biggl\{\exp\biggl(\frac{s^2}{2p}\frac{\beta^4}{2^{|l-k|}}e^{\beta^4/2^{|l-k|+1}}\biggr)\indc{k\neq l}\biggr\}.
    \end{align*}
    Notice that $\beta\rightarrow 0$ in our asymptotic regime and therefore $e^{\beta^4/2^{|l-k|+1}}\rightarrow 1$ uniformly for any $k,l \in \mathcal{G}_n$. Further notice that when $k,l \in \mathcal{G}_n$ satisfy $l\neq k$, we have $|l-k|\geq \floor{c\log\log\log n}$ and therefore when $c$ is sufficiently large,
    \[\frac{s^2\beta^4}{2p2^{|l-k|}}\leq\frac{(2-\epsilon)^2}{2}\frac{\log\log n}{2^{\floor{c\log\log\log n}}} \rightarrow 0.\]
    Thus we have shown that 
    \[\limsup_{n\rightarrow\infty}\mathbb{E}_{k,l}\bigl\{\mathbb{E}_{S,u,T,v}\bigl(e^{\iprod{\theta_1}{\theta_2}}\bigr)\indc{k\neq l}\bigr\} \leq 1.\] 
  Next, we need to show that
  \[\mathbb{E}_{k,l}\bigl[\mathbb{E}_{S,u,T,v}\bigl\{\mathbb{P}_{X\sim N_{p\times n}(\theta_1+\theta_2,I)}(X\in A_n)\exp\bigl(\iprod{\theta_1}{\theta_2}\bigr)\bigr\}\indc{k=l}\bigr] \rightarrow 0.
  \]
We let $\tilde u\in \mathbb{R}^p$ be the vector that equals to $u$ on $S$ and $0$ otherwise and we define $\tilde v$ in the same way for $v$. Then
\begin{align}
  \label{Eq:DelicateDecomp}
      \mathbb{E}_{k,l}&\bigl[\mathbb{E}_{S,u,T,v}\bigl\{\mathbb{P}_{X\sim N_{p\times n}(\theta_1+\theta_2,I)}(X\in A_n)e^{\iprod{\theta_1}{\theta_2}}\bigr\}\indc{k=l}\bigr] \nonumber \\
    &=\mathbb{E}_{k,l}\bigl[\mathbb{E}_{S,u,T,v}\bigl\{\mathbb{P}_{X\sim N_{p\times n}(\theta_1+\theta_2,I)}(X\in A_n)e^{\iprod{\theta_1}{\theta_2}}\indc{\norm{\tilde u+\tilde v}^2\leq (2-\epsilon_2)s}\bigr\}\indc{k=l}\bigr] \nonumber \\
    &\hspace{.5cm}+\mathbb{E}_{k,l}\bigl[\mathbb{E}_{S,u,T,v}\bigl\{\mathbb{P}_{X\sim N_{p\times n}(\theta_1+\theta_2,I)}(X\in A_n)e^{\iprod{\theta_1}{\theta_2}}\indc{\norm{\tilde u+\tilde v}^2> (2-\epsilon_2)s}\bigr\}\indc{k=l}\bigr].
    \end{align}
    We first deal with the second term. When $k=l$, we let $t_0=2^k=2^l$. When $\norm{\tilde u+\tilde v}^2>(2-\epsilon_2)s$, the noncentrality parameter of $\norm{X_1+\dots+X_{t_0}}^2/t_0$ is 
    \[\beta^2\norm{\tilde u+\tilde v}^2> (2-\epsilon)(2-\epsilon_2)\sqrt{p\log\log n}.\]
    By Lemma~\ref{lem:noncentral-chi-square-tail},  on the event that $k=l$ and $\norm{\tilde u+\tilde v}^2>(2-\epsilon_2)s$, we have
    \begin{align*}
    \mathbb{P}_{X\sim N_{p\times n}(\theta_1+\theta_2,I)}(X\in A_n) &\leq \mathbb{P}\left(\chi^2_{p,(2-\epsilon)(2-\epsilon_2)\sqrt{p\log\log n}}\leq p+(2+\epsilon_1)\sqrt{p\log\log n}\right)\\
    &\leq \exp\biggl\{-(1+o(1))\left(\frac{(2-\epsilon)(2-\epsilon_2)-(2+\epsilon_1)}{2}\right)^2\log\log n\biggr\}\\
    &=\left(\frac{1}{\log n}\right)^{(1+o(1))\left(\frac{(2-\epsilon)(2-\epsilon_2)-(2+\epsilon_1)}{2}\right)^2}.
    \end{align*}
    In addition, we also have
    $$\mathbb{P}(k=l)=\frac{1}{\floor{\log_2(\sqrt{n})/\floor{c\log\log\log n}} +1}.$$
    Therefore,
    \begin{align*}
      \mathbb{E}_{k,l}\bigl[\mathbb{E}_{S,u,T,v}\bigl\{&\mathbb{P}_{X\sim N_{p\times n}(\theta_1+\theta_2,I)}(X\in A_n)e^{\iprod{\theta_1}{\theta_2}}\indc{\norm{\tilde u+\tilde v}^2> (2-\epsilon_2)s}\bigr\}\indc{k=l}\bigr] \\
    &\leq \biggl(\frac{1}{\log n}\biggr)^{(1+o(1))\left(\frac{(2-\epsilon)(2-\epsilon_2)-(2+\epsilon_1)}{2}\right)^2}\exp\biggl(\frac{s^2\beta^4}{2p}e^{\beta^4}\biggr)\mathbb{P}(k=l) \\
    &= \frac{\left(\log n\right)^{(1+o(1))\frac{(2-\epsilon)^2}{2}}}{\left(\log n\right)^{(1+o(1))\left[\left(\frac{(2-\epsilon)(2-\epsilon_2)-(2+\epsilon_1)}{2}\right)^2+1\right]}}.
    \end{align*}
    By choosing $\epsilon_1,\epsilon_2 > 0$ sufficiently small, we can ensure that the power of the $\log n$ factor in the denominator is arbitrarily close to $2-2\epsilon+\epsilon^2$, while the power of the $\log n$ factor in the numerator is arbitrarily close to $2-2\epsilon+\epsilon^2/2$. We deduce that
    \[
      \mathbb{E}_{k,l}\bigl[\mathbb{E}_{S,u,T,v}\bigl\{\mathbb{P}_{X\sim N_{p\times n}(\theta_1+\theta_2,I)}(X\in A_n)e^{\iprod{\theta_1}{\theta_2}}\indc{\norm{\tilde u+\tilde v}^2> (2-\epsilon_2)s}\bigr\}\indc{k=l}\bigr] \rightarrow 0.
  \]
Finally, we analyse the first term on the right-hand side of~\eqref{Eq:DelicateDecomp}.  
   By the Cauchy--Schwarz inequality, we first obtain the bound
   \begin{align*}
     \mathbb{E}_{k,l}\bigl[\mathbb{E}_{S,u,T,v}\bigl\{\mathbb{P}_{X\sim N_{p\times n}(\theta_1+\theta_2,I)}&(X\in A_n)e^{\iprod{\theta_1}{\theta_2}}\indc{\norm{\tilde u+\tilde v}^2\leq (2-\epsilon_2)s}\bigr\}\indc{k=l}\bigr] \\
     &\leq \bigl(\mathbb{E} e^{2\iprod{\theta_1}{\theta_2}}\bigr)^{1/2}\sqrt{\mathbb{P}\bigl(\norm{\tilde u+\tilde v}^2\leq (2-\epsilon_2)s\bigr)}.
     \end{align*}
 Now, by construction,
 $$\mathbb{E} e^{2\iprod{\theta_1}{\theta_2}}\leq \mathbb{E}e^{\fnorm{\theta_1}^2+\fnorm{\theta_2}^2}\leq e^{2s\beta^2}=e^{2(2-\epsilon)\sqrt{p\log\log n}}.$$
    To bound $\mathbb{P}\left(\norm{\tilde u+\tilde v}^2\leq (2-\epsilon_2)s\right)$, we condition on the event that $|S\cap T|=a \leq s$ and $|S\setminus T|=|T\setminus S|=s-a$. Denoting the distribution of $a$ by $r$, we then have by Hoeffding's inequality that 
    \begin{align*}
    \mathbb{P}\left(\norm{\tilde u+\tilde v}^2\leq (2-\epsilon_2)s\right) &= \mathbb{E}_{a\sim r}\mathbb{P}\bigl(2(s-a)+4\cdot\mathrm{Bin}(a,1/2)\leq (2-\epsilon_2)s\bigm|a\bigr)\\
                                                                          &\leq \max_{a\leq s}\mathbb{P}\biggl(|\mathrm{Bin}(a,1/2)-a/2|\geq \frac{\epsilon_2}{4}s\biggm|a\biggr) \\
      &\leq \max_{a\leq s} 2e^{-\frac{\epsilon_2^2s^2}{8a}}=2e^{-\epsilon_2^2s/8}.
    \end{align*}
    Therefore,
    \begin{align*}
      \mathbb{E}_{k,l}\bigl[\mathbb{E}_{S,u,T,v}\bigl\{\mathbb{P}_{X\sim N_{p\times n}(\theta_1+\theta_2,I)}(X\in A_n)e^{\iprod{\theta_1}{\theta_2}}&\indc{\norm{\tilde u+\tilde v}^2 \leq (2-\epsilon_2)s}\bigr\}\indc{k=l}\bigr] \\
      &\leq 2e^{2(2-\epsilon)\sqrt{p\log\log n}-\epsilon_2^2s/8}\rightarrow 0.
      \end{align*}
This verifies Condition~2 and hence concludes our lower bound calculation.
  \end{proof}

\begin{proof}[Proof of Theorem \ref{thm:asymp-dense} (upper bound)]
	For the upper bound, define
	\begin{equation}
	\widetilde Y_t:= \sqrt{\frac{t(n-t)}{n}}\left(\frac{1}{t}(X_1+\ldots+X_t)-\frac{1}{n-t}(X_{t+1}+\ldots+X_n)\right).\label{eq:tilde-Y}
	\end{equation}
	Assume that $\rho^2\geq 2\sqrt{(1+\epsilon)p\log\log n}$ for some constant $\epsilon \in (0,1]$.  For $\delta_1, \delta_2 > 0$, consider the test 
	\[\psi:=\mathbbm{1}_{\left\{\max_{t\in\mathcal{T}_{\delta_2}}\|\widetilde Y_t\|_2^2-p\geq 2\sqrt{(1+\delta_1)p\log\log n}\right\}},\]
	where $\mathcal{T}_{\delta_2} :=\mathcal{T}_{\delta_2}^1\cup \mathcal{T}_{\delta_2}^2$ with
	\[\mathcal{T}_{\delta_2}^1:=\left\{1,\floor{1+\delta_2},\floor{(1+\delta_2)^2},\ldots,\floor{(1+\delta_2)^{\floor{\log_{1+\delta_2}(n/2)}}}\right\},\]
	and $\mathcal{T}_{\delta_2}^2=\{n-t:t\in\mathcal{T}_{\delta_2}^1\}$.
	Under the null hypothesis, for a fixed $t$, we have $\norm{\widetilde Y_t}_2^2\sim \chi^2_p$. Therefore, by Lemma \ref{lem:chi-square-tail} a union bound argument, we have
	\begin{align*}
	\sup_{\theta \in \Theta_0(p,n)} \mathbb{P}_{\theta}\biggl(\max_{t\in\mathcal{T}_{\delta_2}}\|\widetilde Y_t\|^2-p & \geq 2\sqrt{(1+\delta_1)p\log\log n}\biggr)\\
	&\leq 2\floor{\log_{1+\delta_2}(n/2)}e^{-(1+o(1))(1+\delta_1)\log\log n}\rightarrow 0.
	\end{align*}
	Under the alternative hypothesis, for $\theta\in\Theta(p,n,s,\rho)$, assume that the true changepoint is at~$t_0$ and that the mean vector before and after the changepoint is $\mu_1$ and $\mu_2$, respectively. Assume $t_0\leq n/2$ without loss of generality. Let $\tilde{t}_0 = \tilde{t}_0(\theta)$ be the closest point in $\mathcal{T}_{\delta_2}$ to $t_0$ such that $\tilde {t}_0\leq t_0$. Then $\tilde{t}_0\leq t_0\leq (1+\delta_2)\tilde{t}_0$. By the assumption that $\rho^2\geq 2\sqrt{(1+\epsilon)p\log\log n}$, we have $\norm{\widetilde Y_{\tilde{t}_0}}_2^2\sim \chi^2_{p,\lambda}$, where 
	\begin{align}
       \nonumber   \lambda = \frac{\tilde{t}_0(n-\tilde{t}_0)}{n}\biggl\|\frac{n-t_0}{n-\tilde{t}_0}(\mu_1-\mu_2)\biggr\|^2 &\geq \frac{\tilde{t}_0(n-t_0)}{n}\norm{\mu_1-\mu_2}^2 \geq \frac{\tilde{t}_0}{t_0}\rho^2 \\
      \label{eq:sig-cons-5-6}                                                                                                            &\geq \frac{\rho^2}{1+\delta_2} \geq 2\sqrt{(1+ \epsilon/2)p\log\log n}
          \end{align}
if we choose $\delta_2 = \epsilon/8$. By Lemma \ref{lem:noncentral-chi-square-tail}, we therefore have
	\begin{align*}
          \sup_{\theta \in \Theta(p,n,s,\rho)} \mathbb{P}_{\theta}\left(\|\widetilde Y_{\tilde {t}_0}\|^2-p\leq 2\sqrt{(1+\delta_1)p\log\log n}\right) &\leq e^{-(1+o(1))(\sqrt{1+\epsilon/2}-\sqrt{1+\delta_1})^2\log\log n} \\
          &\rightarrow 0,
	\end{align*}
as long as $\delta_1 \in (0,\epsilon/2)$.  This completes the proof of the upper bound.
\end{proof}


\begin{proof}[Proof of Theorem \ref{thm:asymp-sparse} (lower bound)]
	As in the proof of Theorem \ref{thm:asymp-dense}, we will allow $n$ to be sufficiently large. Let $f_{0n} := \phi$ be the density function of the standard normal distribution on $\mathbb{R}^{p\times n}$. Define $f_{1n}(x) := \int_{\supp(\nu_n)} \phi(x-\theta) \, d\nu_n(\theta)$, where $\nu_n$ is the distribution of $\theta$ when $\theta$ is generated according to the following sampling process:
	\begin{enumerate}
		\item Uniformly sample a subset $S \subseteq [p]$ of cardinality $s$;
		\item Independently of $S$, generate $k \sim \mathrm{Unif}(\mathcal{G}_n)$ where 
		\begin{align*}
		\mathcal{G}_n := \bigl\{0,&\floor{\log\log\log n},2\floor{\log\log\log n},\ldots, \\
		&\floor{\log\log\log n}\floor{\log_2(\sqrt{n})/\floor{\log\log\log n}}\bigr\};
		\end{align*}
		\item Given $(S,k)$ sampled in the previous steps, define $\theta_{j\ell} :=\frac{\beta}{\sqrt{2^k}}$ for all $(j,\ell)\in S\times [2^k]$ and $\theta_{j\ell}:=0$ otherwise, where $\beta^2 = (1-\epsilon)\log\left(\frac{p\log\log n}{s^2}\right)$ for some $\epsilon \in (0,1)$.
	\end{enumerate}
Similar to our previous arguments, we let $(S,k)$ and $(T,l)$ be independent with distribution $\nu_n$.  Since $\frac{2^k(n-2^k)}{n}\cdot {\frac{\beta^2}{2^k}|S|} \geq \frac{n-\sqrt{n}}{n}(1-\epsilon)s\log\left(\frac{p\log\log n}{s^2}\right)$, we have $\supp(\nu_n)\subseteq\Theta(p,n,s,\rho)$ with $\rho^2=(1-2\epsilon)s\log\left(\frac{p\log\log n}{s^2}\right)$. Therefore, by Lemmas~\ref{lem:lower} and~\ref{lem:chi-square-cal}, it is sufficient to show 
  \[\limsup_{n\rightarrow\infty}\mathbb{E}_{(\theta_1,\theta_2)\sim \nu_n\otimes \nu_n} \exp\left(\iprod{\theta_1}{\theta_2}\right)\leq 1.\]
By the same calculation that leads to (\ref{eq:no-where}), we have
    \[\mathbb{E}_{(\theta_1,\theta_2)\sim \nu_n\otimes \nu_n} \bigl\{\exp\left(\iprod{\theta_1}{\theta_2}\right)\bigr\}\leq \mathbb{E}_{(l,k)}\biggl\{\exp\left(\frac{s^2}{p}e^{\beta^2/2^{|l-k|/2}}\right)\biggr\}.\]
    We split the right hand side above into two terms according to whether or not $l=k$. When $l\neq k$, we have $|l-k|\geq \floor{\log\log\log n}$ due to the definition of $\nu_n$. Therefore,
    \begin{align*}
    \mathbb{E}_{(l,k)}\biggl\{\exp\biggl(\frac{s^2}{p}e^{\beta^2/2^{|l-k|/2}}\biggr)\mathbbm{1}_{\{l\neq k\}}\biggr\} \leq \exp\Biggl(\frac{s^2}{p}\biggl(\frac{p\log\log n}{s^2}\biggr)^{1/2^{\floor{\log\log\log n}/2}}\Biggr)\rightarrow 1,
    \end{align*}
	where we have used $s^2/p\rightarrow 0$ and $(\log\log n)^{1/\log\log n}\rightarrow 1$. When $l=k$, we have
	\begin{align*}
	\mathbb{E}_{(l,k)}\biggl\{\exp\left(\frac{s^2}{p}e^{\beta^2/2^{|l-k|/2}}\right)\mathbbm{1}_{\{l= k\}}\biggr\} \leq \frac{1}{\floor{\log_2(\sqrt{n})/\floor{\log\log\log n}}}\exp\left(\frac{s^2}{p}e^{\beta^2}\right)\rightarrow 0,
	\end{align*}
	according to the definition of $\beta$. Combining the two bounds, we have obtained the desired conclusion.
\end{proof}

\begin{proof}[Proof of Theorem \ref{thm:asymp-sparse} (upper bound)]
Consider the statistic
$$\wt{A}_{t,a}=\sum_{j=1}^p\bigl(\wt{Y}_t^2(j)-\nu_a\bigr)\mathbbm{1}_{\{|\wt{Y}_t(j)|\geq a\}},$$ where the definition of $\wt{Y}_t$ is given by (\ref{eq:tilde-Y}). Recall the definition of $\mathcal{T}_{\delta_2}$ in the upper bound proof of Theorem \ref{thm:asymp-dense}. We then consider the testing procedure
$$\tilde{\psi}:=\mathbbm{1}_{\left\{\max_{t\in\mathcal{T}_{\delta_2}}\wt{A}_{t,a}\geq C^*\left(\sqrt{pe^{-a^2/2}2\log\log n}+2\log\log n\right)\right\}},$$
where $a=\sqrt{2\log(\frac{p\log\log n}{s^2})}$, and the constant $C^*>0$ is taken from Lemma~\ref{lem:trunc-chi-square-tail}.
    Under the null hypothesis, i.e. for any $\theta\in \Theta_0(p,n)$, by Lemma~\ref{lem:trunc-chi-square-tail} and a union bound argument, we have
    \begin{align*}
    \sup_{\theta\in \Theta_0(p,n)}\mathbb{E}_\theta\tilde{\psi}\leq 2\left(\floor{\log_{1+\delta_2}(n/2)}+1\right)\exp(-2\log\log n)\rightarrow 0.
    \end{align*}

    Next, we study $\mathbb{E}_\theta(1-\tilde{\psi})$ under the alternative hypothesis. For $\theta\in\Theta(p,n,s,\rho)$ with $\rho$ satisfying $\rho^2\geq 2(1+\epsilon)s\log\left(\frac{p\log\log n}{s^2}\right)$, we denote the changepoint of $\theta$ by $t_0$, and the mean vector before and after the changepoint by $\mu_1$ and $\mu_2$, respectively. Assume that $t_0\leq n/2$ without loss of generality. Let $\wt{t}_0=\wt{t}_0(\theta)$ be the closest point in $\mathcal{T}_{\delta_2}$ to $t_0$ such that $\wt{t}_0\leq t_0$. Then we have $\wt{t}_0\leq t_0\leq (1+\delta_1)\wt{t}_0$.  Write $\Delta_j := \mathbb{E}\bigl\{\wt{Y}_{\wt{t}_0}(j)\bigr\}$ as shorthand.   By the assumption that $\rho^2\geq 2(1+\epsilon)s\log\left(\frac{p\log\log n}{s^2}\right)$ and the same argument we have used in (\ref{eq:sig-cons-5-6}), we have
$$\sum_{j=1}^p\Delta_j^2\geq\frac{\rho^2}{1+\delta_2}\geq 2(1+\epsilon/2)s\log\left(\frac{p\log\log n}{s^2}\right),$$
by choosing $\delta_2=\epsilon/8$. By (\ref{eq:alt-sig-sharp}) of Lemma~\ref{lem:trunc-chi-square-mean}, for any $\delta_1>0$, there exists $c_1^*=c_1^*(\delta_1)>0$ such that
\begin{align*}
\mathbb{E}_{\theta}\wt{A}_{\wt{t}_0,a} &\geq c_1^*\sum_{j=1}^p\Delta_j^2\mathbbm{1}_{\{|\Delta_j|\geq (1+\delta_1)a\}} \\
&= c_1^*\biggl(\sum_{j=1}^p\Delta_j^2-\sum_{j=1}^p\Delta_j^2\mathbbm{1}_{\{|\Delta_j|< (1+\delta_1)a\}}\biggr) \\
&\geq c_1^*\biggl\{2(1+\epsilon/2)s\log\left(\frac{p\log\log n}{s^2}\right)-(1+\delta_1)^2sa^2\biggr\} \\
&= \frac{c_1^*\epsilon}{2}s\log\left(\frac{p\log\log n}{s^2}\right),
\end{align*}
by choosing $\delta_1$ such that $2(1+\delta_1)^2 = 1 + \epsilon/2$.
For the variance, due to Lemma~\ref{lem:trunc-chi-square-var}, we have
	\[\Var_{\theta}\bigl(\wt{A}_{\wt{t}_0,a}\bigr)=O\left(pa^3e^{-a^2/2}+sa^4\right).\]
Hence, following the same argument used in (\ref{eq:usedlater-2}), we have
\begin{align*}
\sup_{\theta\in\Theta(p,n,s,\rho)}\mathbb{E}_\theta(1-\tilde{\psi}) &\leq \sup_{\theta\in\Theta(p,n,s,\rho)}\mathbb{P}_\theta\biggl(\wt{A}_{\wt{t}_0,a}< C^*\Bigl(\sqrt{pe^{-a^2/2}2\log\log n}+2\log\log n\Bigr)\biggr) \\
&\rightarrow 0,
\end{align*}
which leads to the desired conclusion.
\end{proof}

\subsection{Proofs of results in Section \ref{sec:spatial}}


\begin{proof}[Proof of Theorem \ref{thm:minimax-spatial}]
	For any $\theta\in\Theta(p,n,p,\rho)$, there exist $\mu_1,\mu_2\in\mathbb{R}^p$ and $t\in[n]$, such that $X_1,\ldots,X_t\stackrel{\mathrm{iid}}{\sim} N_p(\mu_1,\Sigma)$ and $X_{t+1},\ldots,X_n\stackrel{\mathrm{iid}}{\sim} N_p(\mu_2,\Sigma)$. The covariance matrix $\Sigma$ admits the eigenvalue decomposition $\Sigma = U\Lambda U^T$ for some orthogonal $U\in\mathbb{R}^{p\times p}$ and $\Lambda = \mathrm{diag}(\lambda) \in \mathbb{R}^{p \times p}$, where $\lambda := (\lambda_1,\ldots,\lambda_p)^T$ and $\lambda_1 \geq \cdots \geq \lambda_p > 0$. Then $U^TX_1,\ldots,U^TX_t\stackrel{\mathrm{iid}}{\sim} N_p(U^T\mu_1,\Lambda)$ and $U^TX_{t+1},\ldots,U^TX_n\stackrel{\mathrm{iid}}{\sim} N_p(U^T\mu_2,\Lambda)$. Moreover, $\|U^T(\mu_1-\mu_2)\|=\|\mu_1-\mu_2\|$, so we can consider a diagonal $\Sigma$ without loss of generality. From now on, we assume that $\Sigma=\Lambda$.
	
	We first derive the upper bound. Consider the testing procedure
	$$\psi=\mathbbm{1}_{\{\max_{t\in\mathcal{T}}\|Y_t\|^2-\Tr(\Sigma) > r\}},$$
	with $r=C\left(\sqrt{\fnorm{\Sigma}^2\log\log(8 n)}+\opnorm{\Sigma}\log\log(8n)\right)$ for some appropriate $C>0$. Then the same argument in the proof of Proposition~\ref{thm:upper} together with Lemma \ref{lem:chi-square-tail} leads to the desired result.
	
	We now derive the lower bound. We first seek to apply Lemmas~\ref{lem:lower} and~\ref{lem:chi-square-cal} and given $\eta > 0$, find a probability measure $\nu$ with $\supp(\mu)\subseteq\Theta(p,n,p,\rho)$ and a universal constant $c > 0$ such that
	\begin{equation}
	\mathbb{E}_{(\theta_1,\theta_2)\sim \nu\otimes\nu}\exp(\iprod{\theta_1}{\theta_2}_{\Sigma^{-1}})\leq 1+\eta,\label{eq:lower-spatial-goal}
	\end{equation}
	whenever $\rho = c\rho_{\Sigma}^*$.  We define $\nu$ to be the distribution of $\theta = (\theta_{j\ell}) \in \Theta(p,n,p,\rho)$, sampled according to the following process:
	\begin{enumerate}
		\item Uniformly sample $k\in\{0,1,2,\ldots,\floor{\log_2(n/2)}\}$;
		\item Independently of $k$, sample $u = (u_1,\ldots,u_p)^T \in\mathbb{R}^p$ with independent coordinates, and with $u_j \sim \text{Unif}(\{-a_j,a_j\})$ for $j\in[p]$;
		\item Given $(k,u)$ sampled in the previous steps, define $\theta_{j\ell}:= 2^{-k/2}u_j$ for all $(j,\ell)\in [p]\times [2^k]$ and $\theta_{j\ell}:=0$ otherwise.
	\end{enumerate}
	If $\theta\sim \nu$, then $\theta\in\Theta(p,n,p,\rho)$ with $\rho^2=\frac{1}{2}\sum_{j=1}^pa_j^2$. Suppose that we independently sample $(k,u)$ and $(l,v)$ from the first two steps and use these to construct $\theta_1$ and $\theta_2$ respectively according to the third step. Then, by direct calculation, we obtain
	$$\iprod{\theta_1}{\theta_2}_{\Sigma^{-1}}= (2^k\wedge 2^l)\frac{1}{\sqrt{2^{k+l}}}\sum_{j=1}^p\frac{u_jv_j}{\lambda_j}=\frac{1}{2^{|k-l|/2}}\sum_{j=1}^p\frac{u_jv_j}{\lambda_j}.$$
	Observe that $u_jv_j\sim \text{Unif}(\{-a_j^2,a_j^2\})$, so
	\begin{align*}
	\mathbb{E}_{(\theta_1,\theta_2)\sim \nu\otimes\nu}&\exp(\iprod{\theta_1}{\theta_2}_{\Sigma^{-1}}) = \mathbb{E}\exp\biggl(\frac{1}{2^{|k-l|/2}}\sum_{j=1}^p\frac{u_jv_j}{\lambda_j}\biggr) \\
	&= \mathbb{E}\prod_{j=1}^p\biggl\{\frac{1}{2}\exp\biggl(\frac{a_j^2}{2^{|k-l|/2}\lambda_j}\biggr) +\frac{1}{2}\exp\biggl(-\frac{a_j^2}{2^{|k-l|/2}\lambda_j}\biggr)\biggr\} \\
	&\leq \mathbb{E}\exp\biggl(\frac{1}{2^{|k-l|}}\sum_{j=1}^p\frac{a_j^4}{\lambda_j^2}\biggr),
	\end{align*}
	where the last inequality above uses the fact that $(e^x+e^{-x})/2 \leq e^{x^2}$. We take $a_j^2= \sqrt{c_1\frac{\lambda_j^4\log\log(8n)}{\|\lambda\|^2}}$ for some sufficiently small $c_1 = c_1(\eta) >0$. Then it can be shown that
	\begin{align*}
	\mathbb{E}\exp&\biggl(\frac{1}{2^{|k-l|}}\sum_{j=1}^p\frac{a_j^4}{\lambda_j^2}\biggr) = \mathbb{E}\exp\left(c_1\frac{\log\log(8n)}{2^{|k-l|}}\right) \leq 1 + \eta
	\end{align*}
	using very similar arguments to those employed in the proof of Proposition~\ref{thm:upper}.  We have therefore established~\eqref{eq:lower-spatial-goal}, which implies the desired lower bound $\rho^2=\frac{1}{2}\sum_{j=1}^pa_j^2\asymp \sqrt{\fnorm{\Sigma}^2\log\log(8n)}$.
	
	We also need to prove the lower bound $\opnorm{\Sigma}\log\log(8n)$. Recall that we have assumed without loss of generality that $\Sigma$ is diagonal with non-increasing diagonal elements. Then in our definition of the parameter space $\Theta^{(t_0)}(p,n,s,\rho)$, if we restrict $\mu_1$ and $\mu_2$ to agree in all coordinates except perhaps the first, then the testing problem is equivalent to testing between $\Theta_0(1,n)$ and $\Theta(1,n,1,\rho)$ with variance $\lambda_1=\opnorm{\Sigma}$. Therefore, the lower bound construction in \cite{gao2017minimax} directly applies here and we obtain the desired rate $\opnorm{\Sigma}\log\log(8n)$.
\end{proof}

\begin{proof}[Proof of Proposition \ref{prop:est-cov}]
	Suppose the index set $\mathcal{D}$ does not include the changepoint. Then, by Lemma~\ref{lem:sample-cov-tail}, we have that for every $x > 0$,
	\begin{equation}
	|\Tr(\wh{\Sigma}_{\mathcal{D}})-\Tr(\Sigma)| \leq 4\biggl(\frac{\sqrt{x}\fnorm{\Sigma}}{\sqrt{n}}+\frac{x\opnorm{\Sigma}}{n}\biggr),\label{eq:good-trace}
	\end{equation}
	with probability at least $1-2e^{-x}$ (notice that substituting $n$ for $n-1$ means we multiply the right-hand side by at most $2$).  We will take $x=p\log(32/\epsilon)$, which guarantees that $e^{-x} \leq \epsilon/32$. Moreover, there exists a universal constant $\widetilde{C}>0$, such that for all $x\geq 1$
	\begin{equation}
	\label{eq:good-op}
	\opnorm{\wh{\Sigma}_{\mathcal{D}}-\Sigma}\leq \widetilde{C}\opnorm{\Sigma}\brac{\sqrt{\frac{p}{n}}\vee \frac{p}{n}\vee \sqrt{\frac{x}{n}}\vee \frac{x}{n}},
	\end{equation}
	with probability at least $1-e^{-x}$ \citep[][Theorem~1]{koltchinskii2017concentration}.  Here we will take $x=p\log(16/\epsilon)$. From this we immediately have the error bounds for $\fnorm{\wh{\Sigma}_{\mathcal{D}}}$ and $\opnorm{\wh{\Sigma}_{\mathcal{D}}}$, because
	$$\bigl|\opnorm{\wh{\Sigma}_{\mathcal{D}}}-\opnorm{\Sigma}\bigr|\leq \opnorm{\wh{\Sigma}_{\mathcal{D}}-\Sigma},$$
	and
	$$\bigl|\fnorm{\wh{\Sigma}_{\mathcal{D}}}-\fnorm{\Sigma}\bigr|\leq \fnorm{\wh{\Sigma}_{\mathcal{D}}-\Sigma}\leq\sqrt{p}\opnorm{\wh{\Sigma}_{\mathcal{D}}-\Sigma}.$$
	Since there is only one changepoint, there exists an event of probability at least $1-\epsilon/8$ on which at least two blocks among $\mathcal{D}_1,\mathcal{D}_2,\mathcal{D}_3$ satisfy~\eqref{eq:good-trace}, and an event of probability at least $1 - \epsilon/8$ on which at least two blocks satisfy~\eqref{eq:good-op}. 
	The result therefore follows on taking $C = 4\log(32/\epsilon) + \widetilde{C}(c^{1/2} \vee 1)\log(16/\epsilon)$. 
\end{proof}

\begin{proof}[Proof of Corollary \ref{cor:upper-cov-spatial}]
	Define a set of good events
	\begin{align*}
	G &:= \Big\{ \bigl|\Tr(\wh{\Sigma})^{(2)}- \Tr(\Sigma)\bigr| \leq \brac{\fnorm{\Sigma} + \opnorm{\Sigma}}/4, \\
	&\hspace{2cm} \bigl|\fnorm{\wh{\Sigma}}^{(2)}-\fnorm{\Sigma}\bigr| \leq \fnorm{\Sigma}/4,  \quad \bigl|\opnorm{\wh{\Sigma}}^{(2)}-\opnorm{\Sigma}\bigr| \leq \opnorm{\Sigma}/4\Big\}.
	\end{align*}
	As a direct application of Proposition~\ref{prop:est-cov}, given $\epsilon > 0$, there exists $c > 0$, depending only on $A$ and~$\epsilon$, such that $\mathbb{P}_{\theta}(G^c)\leq \epsilon/4$ for any $\theta \in \Theta(p,n,p,0)$.  Hence, for $\theta\in\Theta_0(p,n)$, when $C\geq 1$, we have
	\begin{align*}
	\mathbb{E}_{\theta}\psi_{\mathrm{Cov}} 
	&\leq \mathbb{P}_\theta\biggl(\biggl\{\max_{t\in\mathcal{T}}\|Y_t\|^2\!-\!\Tr(\wh{\Sigma})^{(2)} \\
	&\hspace{1.5cm}> C\Bigl(\fnorm{\wh{\Sigma}}^{(2)}\sqrt{\log\log(8n)} + \opnorm{\wh{\Sigma}}^{(2)}\log\log(8n)\Bigr)\biggr\} \bigcap G\biggr) 
	 + \mathbb{P}_\theta(G^c) \\
	&\leq \mathbb{P}_\theta\biggl(\max_{t\in\mathcal{T}}\|Y_t\|^2\!-\!\Tr(\Sigma) \! > \! \frac{C}{2}\Bigl(\fnorm{\Sigma}\sqrt{\log\log(8n)} \! + \!  \opnorm{\Sigma}\log\log(8n)\Bigr)\biggr) \! + \! \frac{\epsilon}{4}.
	\end{align*}
	Therefore, by Theorem~\ref{thm:minimax-spatial}, we can choose $C = C(\epsilon) \geq 1$ large enough that the error under the null is at most $\epsilon/2$.  A very similar argument also applies to $\mathbb{E}_{\theta}(1-\psi_{\mathrm{Cov}})$ for $\theta\in\Theta(p,n,p,\rho)$ with $\rho > 0$: when $\rho^2 \geq 64C\Bigl(\fnorm{\Sigma}\sqrt{\log\log(8n)} \vee \opnorm{\Sigma}\log\log(8n)\Bigr)$ and after increasing $C=C(\epsilon)$ if necessary, the error under the alternative is at most $\epsilon/2$, as required.
\end{proof}

\begin{proof}[Proof of Theorem \ref{thm:sparse-upper-spatial}]
	Recalling the representation of $Y_t(j)$ in~\eqref{Eq:Ytj}, we define an oracle version of $\wt{Y}_t$ in~\eqref{eq:def-tilde-Y} by
	$$\bar{Y}_t := \frac{Y_t-\sqrt{\gamma}W_t\mathbf{1}_p}{\sqrt{1-\gamma}}.$$
	Then 
	\begin{equation}
	\|\wt{Y}_t-\bar{Y}_t\|_{\infty}=\frac{|\text{\sf Median}(Y_t)-\sqrt{\gamma}W_t|}{\sqrt{1-\gamma}}.\label{eq:tilde-Y-bar}
	\end{equation}
	By Lemma \ref{lem:median}, there exist universal constants $C_1,C_2,C_3 > 0$ such that for any $\theta\in\Theta(p,n,s,0)$, we have
	$$\mathbb{P}_{\theta}\biggl\{\frac{|\text{\sf Median}(Y_t)-\sqrt{\gamma}W_t|}{\sqrt{1-\gamma}}>C_1\left(\frac{s}{p}+\sqrt{\frac{1+x}{p}}\right)\biggr\}\leq e^{-C_2x},$$
	as long as $C_3\left(\frac{s}{p}+\sqrt{\frac{1+x}{p}}\right)\leq 1$. Using (\ref{eq:tilde-Y-bar}) and a union bound argument, there exists a universal constant $C_4 > 0$ such that 
	\begin{equation}
	\max_{t\in\mathcal{T}}\|\wt{Y}_t-\bar{Y}_t\|_{\infty} \leq C_4\sqrt{\frac{\log\log(8n)}{p}},\label{eq:diff-tilde-Y-bar}
	\end{equation}
with $\mathbb{P}_{\theta}$-probability at least $1-1/\log^2(en)$, for any $\theta\in\Theta(p,n,s,0)$, under the conditions $s\leq (p\log\log(8n))^{1/5}$
	and $\frac{\log\log(8n)}{p}\leq c$. From now on, the event that (\ref{eq:diff-tilde-Y-bar}) holds is denoted by $G$.
	
	With the above preparations, we can analyze $\mathbb{E}_{\theta}\psi$ for any $\theta\in\Theta_0(p,n)$. Recalling the definition of $g_a(\cdot)$ in (\ref{eq:def-g-a}), we set $C'$ in (\ref{eq:def-g-a}) to be $C_4$ in (\ref{eq:diff-tilde-Y-bar}). Then, on the event $G$, we have $g_a\bigl(\wt{Y}_t(j)\bigr)\leq f_a\bigl(\bar{Y}_t(j)\bigr)$ for $j \in [p]$, and therefore given $\epsilon >0$ we can choose $C = C(\epsilon) > 0$ in the definition of $r$ and $n_0 = n_0(\epsilon) \in \mathbb{N}$ such that 
	\[
	\mathbb{E}_{\theta}\psi_{a,r,C'} \leq \mathbb{E}_{\theta}(\psi\mathbbm{1}_{G}) + \mathbb{P}_{\theta}(G^c) \leq \mathbb{P}_{\theta}\biggl(\max_{t\in\mathcal{T}}\sum_{j=1}^pf_a(\bar{Y}_t(j)) > r\biggr)  + \frac{1}{\log^2(en)} \leq \frac{\epsilon}{2},
	\]
	for $n \geq n_0$, where the last inequality is by the same argument as in~\eqref{eq:type-1-sparse-main} in the proof of Proposition~\ref{thm:upper}.
	
	Now we analyze $\mathbb{E}_{\theta}(1-\psi_{a,r,C'})$ for $\theta\in\Theta(p,n,s,\rho)$.  Recall from the proof of Proposition~\ref{thm:upper} that given any $\theta\in\Theta(p,n,s,\rho)$, we may assume there exists $t_0 \leq n/2$ such that $\frac{t_0(n-t_0)}{n}\|\mu_1-\mu_2\|^2\geq\rho^2$; moreover, there exists a unique $\wt{t}\in\mathcal{T}$ such that $t_0/2 < \wt{t}\leq t_0$, and 
	$$\|\Delta_{\wt{t}}\|^2 = \frac{\wt{t}\|\mu_1-\mu_2\|^2}{2}\geq\frac{t_0\|\mu_1-\mu_2\|^2}{4}\geq\frac{t_0(n-t_0)}{4n}\|\mu_1-\mu_2\|^2\geq\frac{\rho^2}{4}.$$
	We introduce a function
	\begin{align*}
	h_a(x) := \inf\left\{f_a(y): |x-y|\leq \frac{a}{10}\right\} = \begin{cases}
	0, & |x|\leq \frac{9}{10}a, \\
	a^2-\nu_a, & \frac{9}{10}a<|x|\leq \frac{11}{10}a, \\
	\left(|x|-\frac{a}{10}\right)^2 - \nu_a, & |x|> \frac{11}{10}a.
	\end{cases}
	\end{align*}
	To gain some intuition, a plot of the functions $h_1(\cdot)$ and $f_1(\cdot)$ is shown in Figure~\ref{fig:h}. 
	\begin{figure}
		\centering
		\includegraphics[width=10cm,height=7cm]{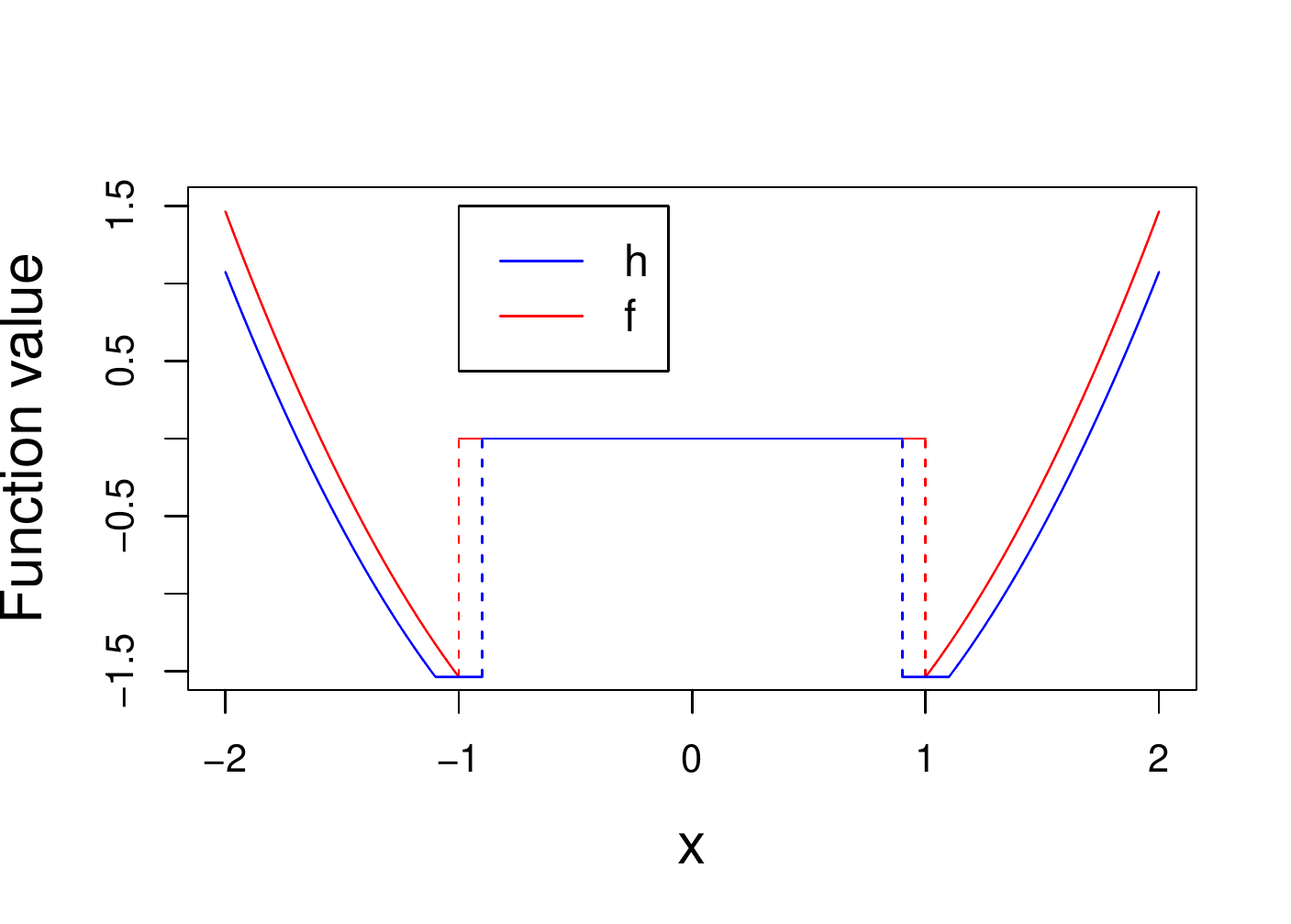}
		\caption{An illustration of the functions $f_a(\cdot)$ and $h_a(\cdot)$ for the special case $a=1$.}
		\label{fig:h}
	\end{figure}
	By reducing $c > 0$ if necessary, we may assume that $2C'\sqrt{\frac{\log\log(8n)}{p}}\leq \frac{a}{10}$, so we have on the event $G$ that $g_a(\wt{Y}_t(j))\geq h_a(\bar{Y}_t(j))$ for $j \in [p]$. Thus
	\begin{align*}
	\mathbb{E}_{\theta}(1-\psi_{a,r,C'})& \leq \mathbb{E}_{\theta}\bigl\{(1-\psi_{a,r,C'})\mathbbm{1}_{G}\bigr\} + \mathbb{P}_{\theta}(G^c)\\
	 &\leq \mathbb{P}_{\theta}\biggl(\sum_{j=1}^ph_a(\bar{Y}_{\wt{t}}(j)) \leq r\biggr) + \frac{1}{\log^2(en)},
	\end{align*}
	and we now control the first term on the right-hand side.  When $\Delta_{\wt{t}}(j) = 0$, we have $\mathbb{E}h_a\bigl(\bar{Y}_{\wt{t}}(j)\bigr) \leq \mathbb{E}f_a\bigl(\bar{Y}_{\wt{t}}(j)\bigr) = 0$.  Moreover, by Lemma~\ref{Lemma:muabounds},
	\begin{align*}
	&-\mathbb{E}h_a(\bar{Y}_{\wt{t}}(j)) \\
	&= 2(\nu_a - a^2)\bigl\{\Phi(11a/10) - \Phi(9a/10)\bigr\} + 2\int_{11a/10}^\infty \bigl\{\nu_a - (x - a/10)\bigr\} \phi(x) \, dx \\
	&\leq \frac{4a^3}{5} \phi(9a/10) + 6a^2\bigl\{1-\Phi(11a/10)\bigr\} \lesssim e^{-a^2/3}.
	\end{align*}
	Next, for $0 < |\Delta_{\wt{t}}(j)| < 8(1-\gamma)^{1/2}a$, we have $\mathbb{E}h_a\bigl(\bar{Y}_{\wt{t}}(j)\bigr) \geq -(\nu_a-a^2) \geq -2a^2$, and by Lemma~\ref{lem:trunc-chi-square-mean}, we have $\mathbb{E}h_a\bigl(\bar{Y}_{\wt{t}}(j)\bigr) \leq \mathbb{E}f_a\bigl(\bar{Y}_{\wt{t}}(j)\bigr) \leq \frac{64\Delta_{\wt{t}}^2(j)}{1-\gamma} + 1 \lesssim a^2$.
	
	Finally, we handle the case where $|\Delta_{\wt{t}}(j)| \geq 8(1-\gamma)^{1/2}a$, and assume without loss of generality that $\Delta_{\wt{t}}(j) \geq 8(1-\gamma)^{1/2}a$.  Observe by Lemma~\ref{Lemma:muabounds} that for $x \geq 4a$, we have
	\[
	(x-a/10)^2 - \nu_a \geq x^2 - \frac{ax}{5} - 3a^2 \geq x^2 - \frac{x^2}{20} - \frac{3x^2}{16} \geq 3x^2/4.
	\]
	Hence
	\begin{align*}
	\mathbb{E}h_a\bigl(\bar{Y}_{\wt{t}}(j)\bigr) &\geq \frac{3}{4} \mathbb{E}\bigl( \bar{Y}_{\wt{t}}(j)^2\mathbbm{1}_{\{\bar{Y}_{\wt{t}}(j) \geq 4a\}}\bigr) - (\nu_a - a^2)\mathbb{P}\bigl(\bar{Y}_{\wt{t}}(j) < 4a\bigr) \\
	&\geq \frac{3\Delta_{\wt{t}}^2(j)}{4(1-\gamma)} \mathbb{P}\bigl(\bar{Y}_{\wt{t}}(j) \geq 4a\bigr) - 3a^2 \mathbb{P}\bigl(\bar{Y}_{\wt{t}}(j) < 4a\bigr) \geq \frac{45\Delta_{\wt{t}}^2(j)}{128(1-\gamma)}.
	\end{align*}
	Summarising then, we have
	$$\mathbb{E}h_a(\bar{Y}_{\wt{t}}(j))\begin{cases}
	\leq 0 \text{ and } \gtrsim -e^{-a^2/3} & \mbox{if $\Delta_{\wt{t}}(j) = 0$,} \\
	\gtrsim -a^2 \text{ and } \lesssim a^2 & \mbox{if $0 < |\Delta_{\wt{t}}(j)| < 8(1-\gamma)^{1/2}a$}, \\
	\geq \frac{45\Delta_{\wt{t}}(j)^2}{128(1-\gamma)} & \mbox{if $|\Delta_{\wt{t}}(j)| \geq 8(1-\gamma)^{1/2}a$}.
	\end{cases}$$
	We now study $\Var \, h_a\bigl(\bar{Y}_{\wt{t}}(j)\bigr)$.  When $\Delta_{\wt{t}}(j) = 0$, we have
	\begin{align*}
	\Var \, h_a\bigl(\bar{Y}_{\wt{t}}(j)\bigr) &\leq \mathbb{E}h_a^2\bigl(\bar{Y}_{\wt{t}}(j)\bigr) \\
	&\leq 2\int_{9a/10}^{\infty} \bigl\{\brac{\nu_a-a^2}^2\vee \brac{(x-a/10)^2-\nu_a}^2\bigr\} \phi(x) \, dx \\
	&\leq 2\int_{9a/10}^{\infty} (\nu_a^2\vee x^2) \phi(x) \, dx \lesssim e^{-a^2/4}.
	\end{align*}
	When $0 < |\Delta_{\wt{t}}(j)| < 2(1-\gamma)^{1/2}a$, assuming that $\Delta_{\wt{t}}(j)>0$ without loss of generality and writing $\theta := \Delta_{\wt{t}}(j)/(1-\gamma)^{1/2}$ as shorthand, we have
	\begin{align*}
	&\Var \, h_a\bigl(\bar{Y}_{\wt{t}}(j)\bigr) 
	\leq \mathbb{E}h_a^2\bigl(\bar{Y}_{\wt{t}}(j)\bigr)\\
	&\leq \biggl(\int_{-\infty}^{-\frac{9a}{10}}+\int_{\frac{9a}{10}}^{\theta+a}+\int_{\theta+a}^{\infty}\biggr) \bigl\{\brac{\nu_a-a^2}^2\vee \brac{(|x|-a/10)^2-\nu_a}^2\bigr\}\phi(x-\theta) \, dx\\
	&\lesssim e^{-a^2/4}+a^4+e^{-a^2/4} \lesssim a^4.
	\end{align*}
	Finally, when $|\Delta_{\wt{t}}(j)| \geq 2(1-\gamma)^{1/2}a$, 
	Let us define a random variable $L := \mathbbm{1}_{\{\bar{Y}_{\wt{t}}(j) \geq 11a/10\}}$.  Then assuming that $\Delta_{\wt{t}}(j)\geq 2(1-\gamma)^{1/2}a$ without loss of generality, we have that 
	\begin{align*}
	&\Var \, h_a\bigl(\bar{Y}_{\wt{t}}(j)\bigr)\\
	 &\hspace{0.2cm}= \E\bigl\{\Var\bigl(h_a\bigl(\bar{Y}_{\wt{t}}(j)\bigr)\bigm| L\bigr)\bigr\}+\Var \bigl\{\E\bigl(h_a\bigl(\bar{Y}_{\wt{t}}(j)\bigr)\bigm|L\bigr)\bigr\}\\
	&\hspace{0.2cm}\leq \Prob(L=0)\mathbb{E}\bigl\{h_a^2\bigl(\bar{Y}_{\wt{t}}(j)\bigr)|L=0\bigr\} + \Prob(L=1)\Var\bigl\{(\bar{Y}_{\wt{t}}(j)-a/10)^2|L=1\bigr\} \\
                &\hspace{0.5cm}+ \Prob(L=0)\Prob(L=1) \\
          &\hspace{0.8cm}\times \Bigl\{\bigl|\E\bigl\{h_a\bigl(\bar{Y}_{\wt{t}}(j)\bigr)|L=0\bigr\}\bigr|+\bigl|\E\bigl((\bar{Y}_{\wt{t}}(j) - a/10)^2 - \nu_a|L=1\bigr)\bigr|\Bigr\}^2.
	\end{align*}
	Now, similar to the proof of Lemma~\ref{lem:trunc-chi-square-var},
	\begin{align*}
	&\bigl|\E\bigl\{h_a\bigl(\bar{Y}_{\wt{t}}(j)\bigr)|L=0\bigr\}\bigr| \leq \nu_a + \mathbb{E}\bigl\{\bar{Y}_{\wt{t}}^2(j)\bigm| \bar{Y}_{\wt{t}}(j) < 11a/10\bigr\} \lesssim \frac{\Delta_{\wt{t}}^2(j)}{1-\gamma}, \\
	&\mathbb{E}\bigl\{h_a^2\bigl(\bar{Y}_{\wt{t}}(j)\bigr)|L=0\bigr\} \leq 2\nu_a^2 + 2\mathbb{E}\bigl\{\bar{Y}_{\wt{t}}^4(j)\bigm| \bar{Y}_{\wt{t}}(j) < 11a/10\bigr\} \lesssim \frac{\Delta_{\wt{t}}^4(j)}{(1-\gamma)^2}, \\
	&\bigl|\E\bigl((\bar{Y}_{\wt{t}}(j) \! - \! a/10)^2 \! - \! \nu_a|L=1\bigr)\bigr| \leq \nu_a + \mathbb{E}\bigl\{\bar{Y}_{\wt{t}}^2(j)\bigm| \bar{Y}_{\wt{t}}(j) \geq 11a/10\bigr\} \lesssim \frac{\Delta_{\wt{t}}^2(j)}{1-\gamma}.
	\end{align*}
	But $\mathbb{P}(L=0) = \Phi\bigl(11a/10 - \Delta_{\wt{t}}(j)/(1-\gamma)^{1/2}\bigr) \leq \Phi\bigl(\frac{-9\Delta_{\wt{t}}(j)}{20(1-\gamma)^{1/2}}\bigr)$.  Finally, we note that 
	\[
	\Prob(L=1)\Var\bigl\{(\bar{Y}_{\wt{t}}(j)-a/10)^2|L=1\bigr\} \leq \Var\bigl\{(\bar{Y}_{\wt{t}}(j)-a/10)^2\bigr\} \lesssim \frac{\Delta_{\wt{t}}^2(j)}{1-\gamma}.
	\]
	These observations allow us to deduce that
	$$\Var\left(h_a(\bar{Y}_{\wt{t}}(j))\right) \lesssim \begin{cases}
	e^{-a^2/4} & \mbox{if $\Delta_{\wt{t}}(j) = 0$,} \\
	a^4 & \mbox{if $0 < |\Delta_{\wt{t}}(j)| < 2(1-\gamma)^{1/2}a$,} \\
	\frac{\Delta_{\wt{t}}(j)^2}{1-\gamma} & \mbox{if $|\Delta_{\wt{t}}(j)| \geq 2(1-\gamma)^{1/2}a$.}
	\end{cases}$$
	The bound on the expectation then implies that
	\[\sum_{j:\Delta_{\wt{t}}(j)=0}\bigl|\E h_a\brac{\bar{Y}_{\wt{t}}(j)}\bigr|\lesssim pe^{-a^2/3}\lesssim p\log\log(8n)\brac{\frac{s^2}{p\log\log(8n)}}^{4/3}\leq s,\]
	where we used the condition $s\leq (p\log\log(8n))^{1/5}$. We deduce similarly to the argument in the proof of Proposition~\ref{thm:upper} that 
	\begin{align*}
	\sum_{j=1}^p \mathbb{E}h_a(\bar{Y}_{\wt{t}}(j)) \geq \frac{\|\Delta_{\wt{t}}\|^2}{4(1-\gamma)}
	\end{align*}
	provided we choose $C=C(\epsilon) > 0$ sufficiently large in the definition of $\rho$.  Moreover,
	$$\sum_{j=1}^p \Var\left(h_a(\bar{Y}_{\wt{t}}(j))\right)\lesssim pe^{-a^2/4} + sa^4 + \frac{\|\Delta_{\wt{t}}\|^2}{1-\gamma}.$$
	By Chebychev's inequality, we deduce that 
	\begin{align*}
	\mathbb{P}_{\theta}\biggl(\sum_{j=1}^ph_a(\bar{Y}_{\wt{t}}(j)) \leq r\biggr) &\leq \mathbb{P}_{\theta}\biggl(\sum_{j=1}^ph_a(\bar{Y}_{\wt{t}}(j)) \leq \frac{\|\Delta_{\wt{t}}\|^2}{8}\biggr) \\
	&\lesssim \frac{pe^{-a^2/4} + sa^4 + \frac{\|\Delta_{\wt{t}}\|^2}{1-\gamma}}{\bigl\{\|\Delta_{\wt{t}}\|^2/(1-\gamma)\bigr\}^2} 
	\lesssim \frac{pe^{-a^2/4} + sa^4 + \frac{\rho^2}{1-\gamma}}{\bigl\{\rho^2/(1-\gamma)\bigr\}^2}. 
	\end{align*}
	Hence, by increasing $n_0 = n_0(\epsilon)$ and $C = C(\epsilon) > 0$ if necessary, we may conclude that $\mathbb{E}_\theta(1-\psi_{a,r,C'}) \leq \epsilon/2$, as required.
\end{proof}

\begin{proof}[Proof of Theorem \ref{thm:sparse-lower-spatial}]
	The proof uses similar arguments to those in the proof of Proposition~\ref{thm:lower}, but instead of establishing (\ref{eq:lower-goal}), we need to show that given $\eta > 0$, we can find a universal constant $c > 0$ such that $\mathbb{E}_{(\theta_1,\theta_2)\sim\nu\otimes\nu}\exp(\iprod{\theta_1}{\theta_2}_{\Sigma(\gamma)^{-1}})\leq 1+\eta$ when $\rho \geq cr_{\Sigma(\gamma)}^*$, where $r_{\Sigma(\gamma)}^*$ denotes the right-hand side of~\eqref{Eq:rstargamma}. Since
	\begin{equation}
	\Sigma(\gamma)^{-1}=\kappa_1(\gamma)I_p - \kappa_2(\gamma)\mathbf{1}_p\mathbf{1}_p^T,\label{eq:inv-cov-gamma}
	\end{equation}
	with $\kappa_1(\gamma)=\frac{1}{1-\gamma}$ and $\kappa_2(\gamma)=\frac{\gamma}{(1-\gamma)(1+(p-1)\gamma)}$, the calculation will be very similar, and essentially our argument replaces $I_p$ in the proof of Proposition~\ref{thm:lower} by $\kappa_1(\gamma)I_p$.
	
	First consider the case when $s\leq\sqrt{p\log\log (8n)}$ and $s\log\left(\frac{ep\log\log(8n)}{s^2}\right)\geq \log\log(8n)$. We define $\nu$ to be the distribution of $\theta$, sampled according to the following process:
	\begin{enumerate}
		\item Uniformly sample a subset $S$ of $[p]$ of cardinality $s$;
		\item Independently, sample $k$ according to a uniform distribution on \\$\{0,1,2,\ldots,\floor{\log_2(n/2)}\}$;
		\item Given $(S,k)$ sampled in the previous steps, define $\theta_{j\ell} := \beta/2^{k/2}$ for all $(j,\ell)\in S\times [2^k]$ and $\theta_{j\ell}:=0$ otherwise, where $\beta > 0$.
	\end{enumerate}
	Suppose that we generate $\theta_1$ and $\theta_2$ independently with distribution $\nu$, where $\theta_1$ is generated from $(S,k)$ and $\theta_2$ comes from $(T,l)$.  By (\ref{eq:inv-cov-gamma}), we have $\iprod{\theta_1}{\theta_2}_{\Sigma(\gamma)^{-1}}\leq \frac{\kappa_1(\gamma)\beta^2}{2^{|l-k|/2}}|S\cap T|$, and thus
	$$\mathbb{E}_{(\theta_1,\theta_2)\sim\nu\otimes\nu}\exp\bigl(\iprod{\theta_1}{\theta_2}_{\Sigma(\gamma)^{-1}}\bigr)\leq \mathbb{E}\exp\left(\frac{\kappa_1(\gamma)\beta^2}{2^{|l-k|/2}}|S\cap T|\right).$$
	Note that we obtain the same formula as (\ref{eq:mgfb3step}) except that the $\beta^2$ in (\ref{eq:mgfb3step}) is replaced by $\kappa_1(\gamma)\beta^2$. This immediately implies that the same argument that bounds (\ref{eq:mgfb3step}) can also be applied here and we obtain the lower bound with the desired rate $\kappa_1(\gamma)^{-1}s\log\left(\frac{ep\log\log(8n)}{s^2}\right)$.
	
	Next we consider the case $s\leq\sqrt{p\log\log (8n)}$ and $s\log\left(\frac{ep\log\log(8n)}{s^2}\right)< \log\log(8n)$. The sampling process for $\theta \sim \nu$ is now:
	\begin{enumerate}
		\item Sample $k$ from a uniform distribution on $\{0,1,2,\ldots,\floor{\log_2(n/2)}\}$;
		\item Given $k$, define $\theta_{j\ell}:= \beta/2^{k/2}$ for all $(j,\ell)\in [s]\times [2^k]$ and $\theta_{j\ell}:=0$ otherwise.
	\end{enumerate}
	Similarly to before, $\iprod{\theta_1}{\theta_2}_{\Sigma(\gamma)^{-1}}\leq \frac{\kappa_1(\gamma)\beta^2}{2^{|l-k|/2}}s$, and thus
	$$\mathbb{E}_{(\theta_1,\theta_2)\sim\nu\otimes\nu}\exp(\iprod{\theta_1}{\theta_2}_{\Sigma(\gamma)^{-1}})\leq \mathbb{E}\exp\left(\frac{\kappa_1(\gamma)\beta^2}{2^{|l-k|/2}}s\right).$$
	We can then set $\kappa_1(\gamma)\beta^2s=c_1\log\log(8n)$ for a sufficiently small $c_1>0$, and apply the same argument as in the proof of \citet[][Proposition~4.2]{gao2017minimax}.  The lower bound follows with rate $\rho^2=s\beta^2\asymp \kappa_1(\gamma)^{-1}\log\log(8n)$.
\end{proof}

\newpage

\appendix

\section{Proofs of results in Section~\ref{SECTEMP}}
\begin{proof}[Proof of Theorem \ref{thm:upper-temporal}]
	For $t \in \big[\lfloor n/2 \rfloor\bigr]$, define $\Gamma_t :=\Cov\bigl((E_1+\ldots +E_t)-(E_{n-t+1}+\ldots +E_n)\bigr)$. Then 
	$$\|Y_t\|^2 \leq \|\Gamma_t\|_{\mathrm{op}}\|\Gamma_t^{-1/2}Y_t\|^2 \leq 2t(1+B)\|\Gamma_t^{-1/2}Y_t\|^2.$$
	Now fix $\theta\in\Theta_0(p,n)$.  Given $\epsilon \in (0,1)$, set $C := C(\epsilon) = 4 + 4\log(4/\epsilon)$.  Since $2t\|\Gamma_t^{-1/2}Y_t\|^2\sim\chi_p^2$, by a union bound and Lemma~\ref{lem:chi-square-tail}, given $\epsilon > 0$, writing $x = \{1 + \log(4/\epsilon)\}\log \log (8n)$, we have
	\begin{align*}
	\mathbb{E}_{\theta,\Sigma}\psi_{\mathrm{Temp}} &\leq \mathbb{P}_{\theta,\Sigma}\biggl(\max_{t\in\mathcal{T}}\|Y_t\|^2 - p > Bp + (1+B)(2\sqrt{px}+2x)\biggr) \\
	&=\mathbb{P}_{\theta,\Sigma}\biggl(\max_{t\in\mathcal{T}}\|Y_t\|^2 > (1+B)(p+2\sqrt{px}+2x)\biggr) \\
	&\leq |\mathcal{T}|\mathbb{P}\bigl(\chi_p^2>p+2\sqrt{px}+2x\bigr) \leq 2\log(en)e^{-x} \leq \frac{\epsilon}{2}.
	\end{align*}
	
	Now, for any $\theta\in\Theta(p,n,s,\rho)$, without loss of generality, we may assume there exists $t_0\in [n/2]$, such that $\frac{t_0(n-t_0)}{n}\|\mu_1-\mu_2\|^2\geq\rho^2$, and a unique $\wt{t}\in\mathcal{T}$ such that $t_0/2 < \wt{t}\leq t_0$.  Thus $2\wt{t}\|\Gamma_{\wt{t}}^{-1/2}Y_{\wt{t}}\|^2 \sim \chi_{p,\delta^2}^2$, with 
	$$\delta^2 = \frac{\wt{t}\|\mu_1-\mu_2\|^2}{2}\geq\frac{t_0\|\mu_1-\mu_2\|^2}{4}\geq\frac{t_0(n-t_0)}{4n}\|\mu_1-\mu_2\|^2\geq\frac{\rho^2}{4}.$$
	Therefore,
	\begin{align*}
	\mathbb{E}_{\theta,\Sigma}(\|Y_{\wt{t}}\|^2)& - p\\
	 &= \delta^2 + \mathbb{E}_{\theta,\Sigma}\biggl(\left\|\frac{(E_1+\ldots +E_{\wt{t}})-(E_{n-\wt{t}+1}+\ldots +E_n)}{\sqrt{2\wt{t}}}\right\|^2\biggr)-p \\
	&= \delta^2 + \frac{1}{2\wt{t}}\Tr(\Gamma_{\wt{t}} - 2\wt{t}I_p) \geq \delta^2 - \frac{p}{2\wt{t}}\opnorm{\Gamma_{\wt{t}} - 2\wt{t}I_p} \geq \delta^2 - Bp,
	\end{align*}
	where the last inequality is by expanding $\Gamma_t$ according to its definition and the condition that $\sum_{j\in[n]\backslash\{i\}}\opnorm{\Cov(E_i,E_j)}\leq B$ for all $i\in[n]$.  Since $C \geq 1/4$, we have $\rho^2\geq 8Bp$, and then
	$$\mathbb{E}_{\theta,\Sigma}(\|Y_{\wt{t}}\|^2) - p \geq \frac{\delta^2}{2}.$$
	Now write $W_t := Y_t - \mathbb{E}_{\theta,\Sigma}Y_t$, so that $W_t \sim N_p\bigl(0,\Gamma_t/(2t)\bigr)$, and find an orthogonal matrix $Q_t \in \mathbb{R}^{p \times p}$ such that $Q_t^\top \Gamma_t Q_t = D_t$, where $D_t$ is diagonal.  Then, with $Z_t \sim N_p(0,I_p)$, we have 
	\begin{align*}
\Var&\left(\|Y_{\wt{t}}\|^2 - p\right) \\
	&= \Var\bigl(\|W_{\wt{t}}\|^2 + 2W_{\wt{t}}^\top \mathbb{E}_{\theta,\Sigma}(Y_{\wt{t}})\bigr) \leq 2 \Var\bigl(\|W_{\wt{t}}\|^2\bigr) + \frac{4}{\wt{t}}\mathbb{E}_{\theta,\Sigma}(Y_{\wt{t}})^\top\Gamma_{\wt{t}}\mathbb{E}_{\theta,\Sigma}(Y_{\wt{t}}) \\
	&\leq \frac{1}{2\wt{t}^2} \Var\bigl(Z_{\wt{t}}^\top D_{\wt{t}}Z_{\wt{t}} \bigr) + \frac{4\delta^2}{\wt{t}}\|\Gamma_{\wt{t}}\|_{\mathrm{op}} = \frac{\|\Gamma_{\wt{t}}\|_{\mathrm{F}}^2}{\wt{t}^2} + \frac{4\delta^2}{\wt{t}}\|\Gamma_{\wt{t}}\|_{\mathrm{op}} \\
	&\leq 4p(1+B)^2 + 8\delta^2(1+B).
	\end{align*}
	Using Chebychev's inequality, we therefore have
	\begin{align*}
	\mathbb{E}_{\theta,\Sigma}&(1-\psi_{\mathrm{Temp}})
	= \mathbb{P}_{\theta,\Sigma}\left(\max_{t\in\mathcal{T}}\|Y_t\|^2-p\leq r\right) \\
	&\leq \mathbb{P}_{\theta,\Sigma}\left(\|Y_{\wt{t}}\|^2-p\leq  \frac{\rho^2}{32}\right) \leq \mathbb{P}_{\theta,\Sigma}\left(\|Y_{\wt{t}}\|^2-p\leq  \frac{\delta^2}{8}\right) \\
	&\leq \frac{2^8p(1+B)^2 + 2^9\delta^2(1+B)}{\delta^4} \leq \frac{2^{12}p(1+B)^2 + 2^{11}\rho^2(1+B)}{\rho^4},
	\end{align*}
	and we can ensure this final term is bounded above by $\epsilon/2$ by increasing $C=C(\epsilon)$ so that $C \geq 2^8/\epsilon^{1/2}$.
\end{proof}

\begin{proof}[Proof of Theorem \ref{thm:lower-temporal}]
	It suffices to prove the result with $\rho_{\mathrm{Temp}}^*$ replaced with $\rho_1^* \vee \rho_2^*$, where $\rho_1^* := (Bp)^{1/2}$ and $\rho_2^* := \bigl[(1+B)\bigl\{\sqrt{p \log \log(8n)} \vee \log \log (8n)\bigr\}\bigr]^{1/2}$.  For the lower bound $\rho_1^*$, fixing $a \in (0,1]$, we define a covariance matrix $\Sigma_0\in\mathbb{R}^{pn\times pn}$, specified by the following three conditions:
	\begin{enumerate}
		\item $\Cov(E_t)=I_p$ for all $t\in[n]$;
		\item $\Cov(E_s,E_t)=aI_p$ for all $1\leq s\neq t\leq n/2$;
		\item $\Cov(E_s,E_t)=0$ for the remaining pairs $s\neq t$.
	\end{enumerate}
	A sufficient condition for $\Sigma_0\in\mathcal{C}(p,n,B)$ is $na/2\leq B$. We also define a covariance matrix $\Sigma_1\in\mathbb{R}^{pn\times pn}$, specified by the following three conditions:
	\begin{enumerate}
		\item $\Cov(E_t)=(a+1)I_p$ for all $1\leq t\leq n/2$ and $\Cov(E_t)=I_p$ for all $n/2<t\leq n$;
		\item $\Cov(E_s,E_t)=aI_p$ for all $1\leq s\neq t\leq n/2$;
		\item $\Cov(E_s,E_t)=0$ for the remaining pairs $s\neq t$. 
	\end{enumerate}
	Let $Z \sim N_p(0,aI_p)$, and let $Q$ denote the conditional distribution of $Z$ given that $\|Z\|^2 \geq pa/2$.  In other words, 
	$$Q(V)=\frac{\int_{\{\mu\in V:\|\mu\|^2\geq pa/2\}} dN_p(0,aI_p)}{\int_{\{\mu:\|\mu\|^2\geq pa/2\}}dN_p(0,aI_p)}.$$
	for any Borel measurable $V\subseteq \mathbb{R}^p$.  We then define $\nu$ to be the distribution of the random $p \times n$ matrix $\theta$ that is generated according to the following sampling process:
	\begin{enumerate}
		\item Sample $\mu\sim Q$;
		\item Let $\theta_t=\mu$ for all $1\leq t\leq n/2$ and $\theta_t=0$ for all $n/2<t\leq n$.
	\end{enumerate}
	Then $\supp(\nu)\subseteq\Theta(p,n,p,\rho)$ with $\rho^2=npa/8$. We also define another distribution $\wt{\nu}$ to be the distribution of $\theta$ when it is generated as follows:
	\begin{enumerate}
		\item Sample $\mu\sim N_p(0,aI_p)$;
		\item Let $\theta_t=\mu$ for all $1\leq t\leq n/2$ and $\theta_t=0$ for all $n/2<t\leq n$.
	\end{enumerate}
	To lower bound $\mathcal{R}(\rho)$, we need to specify several distributions.  We define
	$$\mathbb{P}_0:=\mathbb{P}_{0,\Sigma_0}\quad\text{and}\quad\mathbb{P}_1:=\int_{\supp(\nu)}\mathbb{P}_{\theta,I_{pn}}\, d\nu(\theta).$$
	To bridge the relation between $\mathbb{P}_0$ and $\mathbb{P}_1$, we define
	$$\wt{\mathbb{P}}_1:=\int_{\supp(\wt{\nu})}\mathbb{P}_{\theta,I_{pn}} \, d\wt{\nu}(\theta).$$
	We claim that $\wt{\mathbb{P}}_1=\mathbb{P}_{0,\Sigma_1}$.  To see this, first note that if $X \sim \wt{\mathbb{P}}_1$, then $X \stackrel{d}{=} \theta + E$, where $\theta \sim \wt{\nu}$, $E$ has independent $N(0,1)$ entries and $\theta$ and $E$ are independent.  Since $\theta$ is a linear transformation of the Gaussian vector $\mu$, we deduce that $X$ is Gaussian.  Moreover, $\widetilde{\mathbb{E}}_1(X) = 0$ and
	\begin{align*}
	\Cov(X_s,X_t) &= \Cov\bigl(\widetilde{\mathbb{E}}_1(X_s|\mu),\widetilde{\mathbb{E}}_1(X_t|\mu)\bigr) + \widetilde{\mathbb{E}}_1\bigl\{\Cov(X_s,X_t |\mu)\bigr\} \\
	&= \Cov(\theta_s,\theta_t) + \widetilde{\mathbb{E}}_1\bigl\{\Cov(X_s,X_t |\mu)\bigr\} \\
	&= \left\{ \begin{array}{ll} aI_p \mathbbm{1}_{\{1 \leq s, t \leq n/2\}} + I_p & \mbox{if $s=t$} \\
	aI_p \mathbbm{1}_{\{1 \leq s, t \leq n/2\}} & \mbox{if $s \neq t$.}
	\end{array} \right.
	\end{align*}
	In other words, $\Cov(X) = \Sigma_1$, which establishes our claim.  Hence
	\begin{align*}
	\mathcal{R}(\rho) &\geq \inf_{\psi \in \Psi} \bigl\{\mathbb{E}_0\psi + \mathbb{E}_1(1-\psi)\bigr\} = 1-\TV(\mathbb{P}_0,\mathbb{P}_1) \\
	&\geq 1-\TV(\mathbb{P}_0,\wt{\mathbb{P}}_1)-\TV(\mathbb{P}_1,\wt{\mathbb{P}}_1). 
	\end{align*}
	Now,
	\begin{align*}
	\TV(\mathbb{P}_0,\wt{\mathbb{P}}_1) = \TV(\mathbb{P}_{0,\Sigma_0},\mathbb{P}_{0,\Sigma_1})\leq \frac{3}{2}\fnorm{\Sigma_1^{-1}\Sigma_0-I_{np}}&\leq \frac{3}{2}\fnorm{\Sigma_0-\Sigma_1}\\
	&\leq \frac{3}{2}\sqrt{\frac{npa^2}{2}},
	\end{align*}
	where the first inequality is by \citet[][Theorem~1.1]{devroye2018total} and the second inequality is by the fact that the smallest eigenvalue of $\Sigma_1$ is 1.
	
	For the second term, by the data processing inequality \citep{ali1966general,zakai1975generalization}, we obtain
	\begin{align*}
	\TV(\mathbb{P}_1,\wt{\mathbb{P}}_1) \leq \TV(\nu,\wt{\nu}) \leq 2\int_{\{\mu:\|\mu\|^2< pa/2\}}dN_p(0,aI_p)& = 2\mathbb{P}\left(\chi_p^2<\frac{p}{2}\right) \\
	&\leq 2e^{-p/16},
	\end{align*}
	where the final inequality follows from Lemma~\ref{lem:chi-square-tail}. 
	Given $\epsilon > 0$, we can therefore let $a = a_*B/n$ with $a_* := \sqrt{2}\epsilon/(3D)$, which amounts to choosing $c_{\epsilon,D} = a_*^{1/2}/2$ and $p_\epsilon = 16\log(4/\epsilon)$ to obtain the lower bound $\rho_1^*$. 
	
	The lower bound $\rho_2^*$ is relatively easier. Without loss of generality, we assume $n/\lceil B \rceil$ to be an integer.  We then divide the set $[n]$ into consecutive blocks $J_1,J_2,\ldots,J_{n/\lceil B \rceil}$, each of cardinality $\lceil B \rceil$.  We define a covariance matrix $\bar{\Sigma}\in\mathbb{R}^{pn\times pn}$ according to the following two conditions:
	\begin{enumerate}
		\item $\Cov(E_t)=I_p$ for all $t\in[n]$;
		\item $\Cov(E_s,E_t)=I_p$ for all $s\neq t$ in the same block, and otherwise $\Cov(E_s,E_t)=0$.
	\end{enumerate}
	Since $\lceil B \rceil-1\leq B$, we have $\bar{\Sigma}\in\mathcal{C}(p,n,B)$. Define 
	$$\bar{\Theta}(p,n,p,\rho) := \bigcup_{\ell=0}^{n/\lceil B \rceil}\Theta^{(\ell\lceil B \rceil)}(p,n,p,\rho)\subseteq \Theta(p,n,p,\rho).$$
	Then
	\begin{align*}
	&\mathcal{R}(\rho) \geq \inf_{\psi \in \Psi}\biggl(\sup_{\theta\in\Theta_0(p,n)}\mathbb{E}_{\theta,\bar{\Sigma}}\psi + \sup_{\theta\in\bar{\Theta}(p,n,p,\rho)}\mathbb{E}_{\theta,\bar{\Sigma}}(1-\psi)\biggr) \\
	&= \inf_{\psi \in \Psi}\biggl(\sup_{\theta\in\Theta_0(p,n/\lceil B \rceil)}\mathbb{E}_{\theta,I_{pn/\lceil B \rceil}}\psi + \sup_{\theta\in{\Theta}(p,n/\lceil B \rceil,p,\rho/\sqrt{\lceil B \rceil})}\mathbb{E}_{\theta,I_{pn/\lceil B \rceil}}(1-\psi)\biggr).
	\end{align*}
	In other words, we have constructed a covariance structure which leads to a simpler problem with $n/\lceil B \rceil$ independent observations and signal strength $\rho^2/\lceil B \rceil$. By Proposition~\ref{thm:lower}, this simpler problem has lower bound $\rho^2/\lceil B \rceil\gtrsim \sqrt{p\log\log(8n/\lceil B \rceil)}\vee\log\log(8n/\lceil B \rceil)$, which is equivalent to the rate $\lceil B \rceil\left(\sqrt{p\log\log(8n/\lceil B \rceil)}\vee\log\log(8n/\lceil B \rceil)\right)$ for the original problem. Under the condition $B\leq D\sqrt{n/p}$
	, the result follows. 
\end{proof}

\section{Technical Lemmas}

In this section we give the auxiliary results used in the proofs of the main results. We first state some lemmas of chi-squared tail bounds.
\begin{lemma}[Lemma~1 of \citet{laurent2000adaptive}]\label{lem:chi-square-tail}
	Let $Z_1,\ldots,Z_p \stackrel{\mathrm{iid}}{\sim} N(0,1)$ and let $\lambda_1\geq\lambda_2\geq\cdots\geq\lambda_p\geq 0$. Then, for any $x>0$, we have
	$$\mathbb{P}\Biggl(\sum_{j=1}^p\lambda_jZ_j^2 \geq \sum_{j=1}^p\lambda_j + 2\sqrt{x\sum_{j=1}^p\lambda_j^2}+2\lambda_1x\Biggr)\leq e^{-x},$$
	and
	$$\mathbb{P}\Biggl(\sum_{j=1}^p\lambda_jZ_j^2 \leq \sum_{j=1}^p\lambda_j - 2\sqrt{x\sum_{j=1}^p\lambda_j^2}\Biggr)\leq e^{-x}.$$
\end{lemma}

\begin{lemma}[Lemma~8.1 of \citet{birge2001alternative}]\label{lem:noncentral-chi-square-tail}
Let $Y\sim\chi_{p,\lambda}^2$ be a non-central chi-squared random variable with $p$ degrees of freedom and non-centrality parameter $\lambda\geq 0$. Then, for any $x>0$, we have
$$\mathbb{P}\left(Y\geq (p+\lambda)+2\sqrt{x(p+2\lambda)}+2x\right)\leq e^{-x},$$
and
$$\mathbb{P}\left(Y\leq (p+\lambda)-2\sqrt{x(p+2\lambda)}\right)\leq e^{-x}.$$
\end{lemma}

\begin{lemma}\label{lem:sample-cov-tail}
	Let $X_1,\dots,X_n\stackrel{\mathrm{iid}}{\sim} N_p(\mu, \Sigma)$, and let 
	\[\wh \Sigma:=\frac{1}{n-1}\sum_{i=1}^n(X_i-\bar{X})(X_i-\bar{X})^T,\]
	where $\bar{X} := n^{-1}\sum_{i=1}^n X_i$.  Then for any $x>0$,
	\[\Prob\brac{|\Tr(\wh\Sigma)-\Tr(\Sigma)|\geq 2\Fnorm{\Sigma}\sqrt{\frac{x}{n-1}}+2\opnorm{\Sigma}\frac{x}{n-1}}\leq 2e^{-x}.\]
\end{lemma}
\begin{proof}
	After an orthogonal transformation, we may assume without loss of generality that $\Sigma$ is diagonal, with non-negative diagonal entries $\lambda_1,\dots,\lambda_p$, say. Then 
	\[\Tr(\wh\Sigma)=\frac{1}{n-1}\sum_{j=1}^p\sum_{i=1}^n\bigl\{X_i(j)-\bar X(j)\bigr\}^2.\]
	Since for $\lambda_j \neq 0$, we have $\sum_{i=1}^n \bigl\{X_i(j)-\bar X(j)\bigr\}^2/\lambda_j \sim \chi^2_{n-1}$, independently for $j\in[p]$, we have
	\[\Tr(\wh\Sigma)\stackrel{\mathrm{d}}{=} \frac{1}{n-1}\sum_{i=1}^{n-1}\sum_{j=1}^{p} \lambda_jZ_{ij}^2,\]
	where $\{Z_{ij}\}_{i \in [n-1],j \in [p]} \stackrel{\mathrm{iid}}{\sim} N(0,1)$. Then Lemma~\ref{lem:chi-square-tail} implies the result.
\end{proof}

The next four lemmas are properties of the truncated non-central chi-squared distribution.  Recall that $\nu_a = \mathbb{E}\bigl(Z^2 \bigm| |Z| \geq a\bigr)$, where $Z \sim N(0,1)$.  
\begin{lemma}\label{Lemma:muabounds}
	The function $a \mapsto \nu_a$ is strictly increasing on $[0,\infty)$, so that $\nu_a \leq \nu_1 \leq 3$ for all $a \in [0,1]$, and $\nu_a\leq 6$ for all $a\in[0,2]$.  Moreover, the function $a \mapsto \nu_a/a^2$ is strictly decreasing on $(0,\infty)$, so that $\nu_a/a^2 \leq \nu_1 \leq 3$ for all $a \geq 1$, and we also have $\nu_a\leq a^2+2$ for all $a > 0$.  Finally, writing $\gamma_a := \mathbb{E}\bigl(Z^4 \bigm| |Z| \geq a\bigr)$, where $Z \sim N(0,1)$, the function $a \mapsto \gamma_a/a^4$ is strictly decreasing on $(0,\infty)$, so that $\gamma_a/a^4 \leq \gamma_1 \leq 11$ for all $a \geq 1$.
\end{lemma}
\begin{proof}
	First note that $\nu_a = \int_a^\infty x^2 \phi(x) \, dx/\bar{\Phi}(a)$, where $\phi$ denotes the standard normal density function, and where $\bar{\Phi}(a) := \int_a^\infty \phi(x) \, dx$.  Hence, for any $a > 0$, we have
	\[
	\frac{d}{da} \log \nu_a = \frac{-a^2\phi(a)}{\int_{a}^\infty x^2\phi(x) \, dx} + \frac{\phi(a)}{\bar{\Phi}(a)}=\frac{\phi(a)\int_{a}^\infty (x^2-a^2)\phi(x) \, dx}{\bar{\Phi}(a)\int_{a}^\infty x^2\phi(x) \, dx} > 0,
	\]
	which proves the first claim.  For the second claim, let
	\[
	g(a) := \log \frac{\nu_a}{a^2} = - 2 \log a + \log \int_a^\infty x^2 \phi(x) \, dx - \log \bar{\Phi}(a).
	\]
	Then
	\begin{align*}
	g'(a) &= -\frac{2}{a} - \frac{a^2\phi(a)}{\int_a^\infty x^2 \phi(x) \, dx} + \frac{\phi(a)}{\bar{\Phi}(a)} \\
	&= \frac{-2\bar{\Phi}(a) \int_a^\infty x^2 \phi(x) \, dx - a^3 \phi(a)\bar{\Phi}(a) + a\phi(a)\int_a^\infty x^2 \phi(x) \, dx}{a\bar{\Phi}(a)\int_a^\infty x^2 \phi(x) \, dx}.
	\end{align*}
	The denominator of this expression is positive, and, after integrating by parts and writing $h(a) := \phi(a)/\bar{\Phi}(a)$, the numerator is
	\begin{align*}
	-a\phi(a)\bar{\Phi}(a) - 2\bar{\Phi}(a)^2 &- a^3\phi(a)\bar{\Phi}(a) + a^2 \phi(a)^2 \\
	= &\bar{\Phi}(a)^2\{-ah(a) - 2 - a^3h(a) + a^2h(a)^2\} \leq -2\bar{\Phi}(a)^2,
	\end{align*}
	where the final inequality uses the standard Mills ratio bound $h(a) \leq a + 1/a$ for $a > 0$ \citep{Gordon1941}.  This proves the second claim that $a \mapsto \nu_a/a^2$ is strictly decreasing on $(0,\infty)$ and $\nu_a/a^2 \leq \nu_1 \leq 3$ for all $a \geq 1$.  
  For the next claim, note that
    \begin{equation}
      \label{Eq:nuabound}
          \nu_a = 1 + \frac{a\phi(a)}{\bar{\Phi}(a)} \leq a^2 + 2
        \end{equation}
        for all $a > 0$, by the same Mills ratio bound as above.  The final claim follows using very similar arguments, and is omitted for brevity. 
     \end{proof}

\begin{lemma}\label{lem:trunc-chi-square-tail}
	Let $Z_1,\ldots,Z_p \stackrel{\mathrm{iid}}{\sim} N(0,1)$. Then there exists a universal constant $C^*>0$ such that for any $a>0$ and $x>0$, we have
	$$\mathbb{P}\biggl(\sum_{j=1}^p(Z_j^2-\nu_a)\mathbbm{1}_{\{|Z_j|\geq a\}} \geq C^*\Bigl(\sqrt{pe^{-a^2/2}x}+x\Bigr)\biggr) \leq e^{-x}.$$
        In fact, we may take $C^* = 9$.
\end{lemma}
\begin{proof}
    Consider $X=(Z^2-\nu_a)\mathbbm{1}_{\{|Z|\geq a\}}$ with $Z\sim N(0,1)$, and we first derive a bound for the moment generating function $\mathbb{E}(e^{\lambda X})$. Since $\mathbb{E}(X)=0$, we have
    	\[\mathbb{E}(e^{\lambda X})=\mathbb{E}\bigl\{1+\lambda X+(e^{\lambda X}-1-\lambda X) \bigr\} =1+\mathbb{E}\bigl(e^{\lambda X}-1-\lambda X\bigr).\]
By the deterministic bound
    	\[
    	e^x-1-x \leq 
    	\begin{cases}
    	(-x)\wedge x^2,         & \text{if $x<0$,} \\
    	x^2, & \text{if $0\leq x\leq 1$,}\\
    	e^x,            & \text{if $x\geq 1$,}
    	\end{cases}
    	\]
    	we have
    	\[\mathbb{E}\bigl(e^{\lambda X}-1-\lambda X\bigr)\leq \lambda^2\mathbb{E}\bigl(X^2\mathbbm{1}_{\{X<0\}}\bigr)+\lambda^2\mathbb{E}\bigl(X^2\mathbbm{1}_{\{0\leq X\leq 1/\lambda\}}\bigr)+\mathbb{E}\bigl(e^{\lambda X}\mathbbm{1}_{\{X> 1/\lambda\}}\bigr),\]
	and we will bound the three terms separately.
    	By Lemma~\ref{Lemma:muabounds}, the first term is bounded as
    	\[\mathbb{E}\bigl(X^2\mathbbm{1}_{\{X<0\}}\bigr) \leq \mathbb{P}(X<0)(\nu_a-a^2)^2\leq \mathbb{P}(|Z|\geq a)(\nu_a-a^2)^2\leq 4\exp(-a^2/2).\]
    	To bound the second term, note that by a change of variable argument, 
    	\[\int_{\nu_a}^{\infty}(x-\nu_a)^2e^{-x/2} \, dx=e^{-\nu_a/2}\int_{0}^{\infty}y^2e^{-y/2} \, dy = 16e^{-\nu_a/2}.\] 
    	Letting $p(\cdot)$ be the density function of $\chi^2_1$, the second term is therefore bounded as
    	\begin{align*}
          \mathbb{E}\bigl(X^2\mathbbm{1}_{\{0<X\leq 1/\lambda\}}\bigr)\leq \mathbb{E}\bigl(X^2\mathbbm{1}_{\{X>0\}}\bigr)= \int_{\nu_a}^{\infty}(x-\nu_a)^2p(x) \, dx &\leq \frac{16}{\sqrt{2\pi \nu_a}} \exp(-\nu_a/2) \\
          &\leq \frac{16}{\sqrt{2\pi} \exp(-\nu_a/2)}.\end{align*}
For the third term, for $\lambda\leq 1/4$, we have
    	\begin{align*}
    	\mathbb{E}\bigl(e^{\lambda X}\mathbbm{1}_{\{X> 1/\lambda\}}\bigr)&=e^{-\lambda\nu_a}\int_{\nu_a+\frac{1}{\lambda}}^{\infty}e^{\lambda x}p(x) \, dx\leq \frac{e^{-\lambda \nu_a}e^{-(\frac{1}{2}-\lambda)(\frac{1}{\lambda}+\nu_a)}}{\sqrt{2\pi(\nu_a+1/\lambda)}(1/2 - \lambda)}\\
    	&\leq \frac{4e}{\sqrt{2\pi}} e^{-\nu_a/2}e^{-\frac{1}{2\lambda}}\leq \frac{19}{2} \lambda^2e^{-\nu_a/2},
    	\end{align*}
    	where in the last step we have used the facts that $\lambda\leq 1/4$ and $\lambda \mapsto e^{-\frac{1}{2\lambda}}/\lambda^2$ is increasing. Hence, when $\lambda\leq 1/4$, we have
    	\[\mathbb{E}e^{\lambda X}\leq 1+ 20\lambda^2e^{-a^2/2} \leq\exp\bigl(20\lambda^2e^{-a^2/2}\bigr).\]
	Then, for any $t>0$, we have
	\begin{align*}
	\mathbb{P}\biggl(\sum_{j=1}^p(Z_j^2-\nu_a)\mathbbm{1}_{\{|Z_j|\geq a\}} >t\biggr) &\leq \inf_{\lambda \leq 1/4} e^{-\lambda t}\bigl(\mathbb{E}e^{\lambda X}\bigr)^p \leq \inf_{\lambda \leq 1/4} \exp\bigl(-\lambda t + 20p\lambda^2e^{-a^2/2}\bigr) \\
	&= \exp\biggl\{-\biggl(\frac{t^2e^{a^2/2}}{80p}\wedge\frac{t}{8}\biggr)\biggr\}.\end{align*}
	Setting $t= 9\bigl(\sqrt{pe^{-a^2/2}x}+x\bigr)$, we obtain the desired result.
    \end{proof}

\begin{lemma}\label{lem:trunc-chi-square-mean}
	Let $Y\sim N(\theta,1)$.  Then there exists a universal constant $C>0$ such that for every $a \geq 1$,
	\begin{equation}
	\mathbb{E}\bigl\{(Y^2-\nu_a)\mathbbm{1}_{\{|Y|\geq a\}}\bigr\} \begin{cases}
	=0, & \mbox{if $\theta=0$,} \\
	\in [0,C^2a^2+1], & \mbox{if $|\theta|<Ca$}, \\
	\geq \theta^2/2, & \mbox{if $|\theta| \geq Ca$.}
	\end{cases}\label{eq:alt-sig-con}
	\end{equation}
	In fact, we may take $C = 8$. Moreover, for any $\delta>0$, there exist constants $c_1^*,C_1^*>0$, such that as long as $|\theta|\geq (1+\delta)a$ and $a>C_1^*$, we have
	\begin{equation}
	\mathbb{E}\bigl\{(Y^2-\nu_a)\mathbbm{1}_{\{|Y|\geq a\}}\bigr\}\geq c_1^*\theta^2.\label{eq:alt-sig-sharp}
	\end{equation}
\end{lemma}
\begin{proof}
	The fact that $\mathbb{E}\bigl\{(Y^2-\nu_a)\mathbbm{1}_{\{|Y|\geq a\}}\bigr\} = 0$ when $\theta=0$ follows by definition of~$\nu_a$.  
	To analyze the case $|\theta|\geq Ca$, observe that by Cauchy--Schwarz, Lemma~\ref{Lemma:muabounds} and Chebychev's inequality, for all $a \in \bigl[1,(\theta^2+1)^{1/2}\bigr)$,
	\begin{align*}
	\mathbb{E}&\bigl\{(Y^2-\nu_a)\mathbbm{1}_{\{|Y| < a\}}\bigr\} \\
	&\leq \sqrt{\mathbb{E}(Y^4) + \nu_a^2}\sqrt{\mathbb{P}(|Y|<a)} \leq \sqrt{\theta^4+6\theta^2+3+9a^4} \sqrt{\mathbb{P}(Y^2<a^2)} \\
	&\leq (\theta^2+3+3a^2)\frac{\sqrt{\Var(Y^2)}}{\theta^2+1-a^2} = (\theta^2+3+3a^2)\frac{\sqrt{2(1+2\theta^2)}}{\theta^2+1-a^2}.
	\end{align*}
	Therefore, for all $a \in \bigl[1,(\theta^2+1)^{1/2}\bigr)$,
	\begin{align*}
	\mathbb{E}\bigl\{(Y^2-\nu_a)\mathbbm{1}_{\{|Y| \geq a\}}\bigr\} &= \theta^2 + 1 - \nu_a - \mathbb{E}\bigl\{(Y^2-\nu_a)\mathbbm{1}_{\{|Y| < a\}}\bigr\} \\
	&\geq \theta^2+1-3a^2-(\theta^2+3+3a^2)\frac{\sqrt{2(1+2\theta^2)}}{\theta^2+1-a^2}.
	\end{align*}
	Thus, for $C > 1$ and $1 \leq a \leq |\theta|/C$,
	\begin{align*}
	\mathbb{E}\bigl\{(Y^2-\nu_a)\mathbbm{1}_{\{|Y| \geq a\}}\bigr\}& \geq \theta^2 - \frac{3\theta^2}{C^2} - \frac{(1+6/C^2)\theta^2 \times 3|\theta|}{(1-1/C^2)\theta^2} \\
	&\geq \theta^2\biggl(1 - \frac{3}{C^2} - \frac{(1+6/C^2)3/C}{(1-1/C^2)}\biggr),
	\end{align*}
	which is at least $\theta^2/2$, provided that $C \geq 8$.  Finally, when $0<|\theta|<Ca$, we have
	$$\mathbb{E}\bigl\{(Y^2-\nu_a)\mathbbm{1}_{\{|Y| \geq a\}}\bigr\} \leq \mathbb{E}(Y^2) =\theta^2+1\leq C^2a^2+1,$$
	while the fact that $\mathbb{E}\bigl\{(Y^2-\nu_a)\mathbbm{1}_{\{|Y| \geq a\}}\bigr\} \geq 0$ follows because the expression inside the expectation is stochastically increasing in $|\theta|$.
	
Finally, we prove (\ref{eq:alt-sig-sharp}). Given (\ref{eq:alt-sig-con}), we only need to consider $(1+\delta)a\leq |\theta|<Ca$. Note that $\mathbb{E}\bigl\{(Y^2-\nu_a)\mathbbm{1}_{\{|Y|\geq a\}}\bigr\}=\mathbb{P}(|Y|\geq a)\bigl\{\mathbb{E}\bigl(Y^2\bigm||Y|\geq a\bigr)-\nu_a\bigr\}$. The condition $|\theta|\geq (1+\delta)a$ implies that
        \[
          \mathbb{P}(|Y|\geq a) = \bar{\Phi}(a - \theta) + \Phi\bigl(-(a+\theta)\bigr) \geq 1/2
        \]
        and
	\begin{align*}\mathbb{E}\bigl(Y^2\bigm| |Y|\geq a\bigr) &\geq \bigl\{\mathbb{E}\bigl(Y\bigm| |Y|\geq a\bigr)\bigr\}^2 \\
          &= \biggl(\theta + \frac{\int_{a-\theta}^\infty z \phi(z) \, dz + \int_{-\infty}^{-(a+\theta)} z \phi(z) \, dz}{\mathbb{P}(|Y| \geq a)}\biggr)^2 \geq \theta^2.\end{align*}
	Therefore, $\mathbb{E}\bigl\{(Y^2-\nu_a)\mathbbm{1}_{\{|Y|\geq a\}}\bigr\}\geq (\theta^2-\nu_a)/2$.  From~\eqref{Eq:nuabound}, we see that for $a \geq \sqrt{2/\delta}$, we have
        \[
          \mathbb{E}\bigl\{(Y^2-\nu_a)\mathbbm{1}_{\{|Y|\geq a\}}\bigr\} \geq \frac{1}{2}\bigl\{\theta^2 - (1 + \delta)a^2\bigr\} \geq \frac{\theta^2}{2}\biggl(1 - \frac{1}{1+\delta}\biggr) = \frac{\delta}{2(1+\delta)}\theta^2,
        \]
        as required.
\end{proof}

\begin{lemma}\label{lem:trunc-chi-square-var}
	Let $Y\sim N(\theta,1)$.  Then there exists a universal constant $C_1\geq 1$ such that
	$$\Var\bigl\{(Y^2-\nu_a)\mathbbm{1}_{\{|Y|\geq a\}}\bigr\} \leq \begin{cases}
	C_1a^3e^{-a^2/2} & \mbox{if $\theta=0$,} \\
	C_1a^4 & \mbox{if $0<|\theta|< 2a$,} \\
	C_1\theta^2 & \mbox{if $|\theta|\geq 2a$,}
	\end{cases}$$
	as long as $a \geq 1$.
\end{lemma}
\begin{proof}
	Let $a \geq 1$.  When $\theta=0$, we have
	\begin{align*}
	\Var\bigl\{(Y^2-\nu_a)\mathbbm{1}_{\{|Y|\geq a\}}\bigr\} &\leq \mathbb{E}\bigl\{(Y^2-\nu_a)^2\mathbbm{1}_{\{|Y|\geq a\}}\bigr\} =\int_{|x|\geq a} \! (x^2-\nu_a)^2\phi(x) \, dx \\
	&\leq 2\bar{\Phi}(a)(\gamma_a - \nu_a^2) \leq 8a^3e^{-a^2/2},
	\end{align*}
	by~Lemma~\ref{Lemma:muabounds}. 
	For $\theta \neq 0$, we may write
	\begin{align}
	\Var&\bigl\{(Y^2-\nu_a)\mathbbm{1}_{\{|Y|\geq a\}}\bigr\} \nonumber \\
	&= \mathbb{E}\bigl[\Var\bigl\{(Y^2-\nu_a)\mathbbm{1}_{\{|Y|\geq a\}}\bigm|\mathbbm{1}_{\{|Y|\geq a\}}\bigr\}\bigr] \nonumber \\
	&\nonumber\hspace{3cm}+ \Var\bigl[\mathbb{E}\bigl\{(Y^2-\nu_a)\mathbbm{1}_{\{|Y|\geq a\}}\bigm|\mathbbm{1}_{\{|Y|\geq a\}}\bigr\}\bigr] \nonumber \\
	\label{Eq:Decomp}&= \mathbb{P}(|Y|\geq a)\Var\bigl(Y^2\bigm| |Y|\geq a\bigr) \\
	&\nonumber\hspace{3cm}+ \mathbb{P}(|Y|< a)\mathbb{P}(|Y|\geq a)\bigl\{\mathbb{E}\bigl(Y^2-\nu_a\bigm||Y|\geq a\bigr)\bigr\}^2.
	\end{align}
	Now
	\begin{align}
	\nonumber \mathbb{P}(|Y|\geq a)\Var(Y^2\bigm||Y|\geq a)&\leq \mathbb{E}\bigl\{\Var(Y^2 \bigm|\mathbbm{1}_{\{|Y|\geq a\}})\bigr\} \\
	\label{Eq:Term1} &\leq \Var(Y^2)=2(1+2\theta^2).
	\end{align}
	Moreover, writing $Y=\theta+X$ where $X \sim N(0,1)$, we have by Lemma~\ref{Lemma:muabounds} and the fact that $\mathbb{E}(X^2|X \geq b) = 1 + b\phi(b)/\bar{\Phi}(b) \leq 1$ for $b \leq 0$ that
	\begin{align}
	\label{Eq:Term2}
	\mathbb{E}\bigl(Y^2-\nu_a\bigm||Y|\geq a\bigr) &\leq 2\mathbb{E}\bigl( \theta^2+X^2 \bigm| |\theta+X| \geq a \bigr) \nonumber \\
	&\leq 2\theta^2+2 \mathbb{E}\bigl(X^2 \bigm| X\geq a-\theta \bigr)\vee 2\mathbb{E}\bigl(X^2 \bigm| X\leq -a-\theta \bigr) \nonumber \\
	&\leq 2\theta^2+2\mathbb{E}\bigl( X^2\bigm| X\geq a+|\theta|\bigr)\leq 2\theta^2+6(a+|\theta|)^2.
	\end{align}
	Moreover,
	\[
	\mathbb{E}\bigl(\nu_a - Y^2\bigm||Y|\geq a\bigr) \leq \nu_a \leq 3a^2.
	\]
	For the first case when $0<|\theta| < 2a$, the result follows from~\eqref{Eq:Decomp},~\eqref{Eq:Term1},~\eqref{Eq:Term2} and Lemma~\ref{Lemma:muabounds}.  For the second case when $|\theta|\geq 2a$, notice that we also have by Chebychev's inequality that $\mathbb{P}(|Y|< a)\leq \frac{2(1+2\theta^2)}{(\theta^2+1-a^2)^2}\leq 18/\theta^{2}$ and the result follows from combining this with~\eqref{Eq:Decomp},~\eqref{Eq:Term1} and~\eqref{Eq:Term2}.
\end{proof}

The following lemma is a direct consequence of results in \cite{tsybakov2009introduction}. We include the proof for completeness.

\begin{lemma}\label{lem:lower}
	Let $\Theta_0$ and $\Theta_1$ denote general parameter spaces, and consider a family of distribution $\{\mathbb{P}_{\theta}\}_{\theta\in\Theta}$, where $\Theta:=\Theta_0\cup\Theta_1$. Let $\nu_0$ and $\nu_1$ be two distributions supported on $\Theta_0$ and $\Theta_1$ respectively.  For $r \in \{0,1\}$, define $\mathbb{Q}_r$ to be the marginal distribution of the random variable $X$ generated hierarchically according to $\theta \sim \nu_r$ and $X|\theta \sim \mathbb{P}_\theta$.  
	Then
	$$\inf_{\psi \in \Psi}\biggl\{\sup_{\theta\in\Theta_0}\mathbb{E}_{\theta}\psi(X) + \sup_{\theta\in\Theta_1}\mathbb{E}_{\theta}\bigl(1-\psi(X)\bigr)\biggr\}\geq \max\left\{\frac{1}{2}\exp\left(-\alpha\right),1-\sqrt{\frac{\alpha}{2}}\right\},$$
	where $\alpha:=\chi^2(\mathbb{Q}_0 \| \mathbb{Q}_1)$.
\end{lemma}
\begin{proof}
	In a slight abuse of notation, we use $\mathbb{E}_0$ and $\mathbb{E}_1$ to denote expectations with respect to $\mathbb{Q}_0$ and $\mathbb{Q}_1$ respectively. 
	Then 
	\begin{align*}
	\inf_{\psi \in \Psi}\biggl\{\sup_{\theta\in\Theta_0}\mathbb{E}_{\theta}\psi(X) + \sup_{\theta\in\Theta_1}\mathbb{E}_{\theta}\bigl(1\!-\!\psi(X)\bigr)\biggr\} &\geq \inf_{\psi \in \Psi} \biggl\{\mathbb{E}_{0}\psi(X) + \mathbb{E}_{1}\bigl(1\!-\!\psi(X)\bigr)\biggr\}\\
	&= 1-\TV(\mathbb{Q}_0,\mathbb{Q}_1).
	\end{align*}
	The result then follows from elementary bounds on the total variation distance given in equations $(2.25)$, $(2.26)$ and Lemma~2.5 of \cite{tsybakov2009introduction}.
      \end{proof}
      The following is referred to in the main text as the truncated second moment method.
        \begin{lemma}
          \label{Lemma:TruncatedSM}
          For $n \in \mathbb{N}$, let $f_{0n}$ and $f_{1n}$ denote densities on a measure space $(\mathcal{X}_n,\mathcal{A}_n,\mu_n)$.  Suppose that there exists $A_n \in \mathcal{A}_n$ such that
          \begin{enumerate}
          \item $\mathbb{E}_{X \sim f_{0n}}\bigl\{\frac{f_{1n}(X)}{f_{0n}(X)}\mathbbm{1}_{\{X \in A_n\}}\bigr\} \rightarrow 1$;
          \item $\mathbb{E}_{X \sim f_{0n}}\bigl\{\bigl(\frac{f_{1n}(X)}{f_{0n}(X)}\bigr)^2\mathbbm{1}_{\{X \in A_n\}}\bigr\} \rightarrow 1$.
            \end{enumerate}
            Then $f_{1n}(X)/f_{0n}(X)$ converges to $1$ in $f_{0n}$-probability.
          \end{lemma}
          \begin{proof}
            Let $Y_n := \frac{f_{1n}(X)}{f_{0n}(X)}\mathbbm{1}_{\{X \in A_n\}}$.  Then $\mu_n := \mathbb{E}_{X \sim f_{0n}}(Y_n) \rightarrow 1$ and 
            \[
              \mathrm{Var}_{X \sim f_{0n}}(Y_n) = \mathbb{E}_{X \sim f_{0n}}(Y_n^2) - \mu_n^2 \rightarrow 0
            \]
            as $n \rightarrow \infty$.  Now let $Z_n := f_{1n}(X)/f_{0n}(X)$.  Then by Chebychev's inequality that, given $\epsilon \in (0,1)$,
            \begin{align*}
              \mathbb{P}_{X \sim f_{0n}}(|Z_n - 1| > \epsilon) \leq \mathbb{P}_{X \sim f_{0n}}(|Y_n - 1| > \epsilon) &\leq \mathbb{P}_{X \sim f_{0n}}(|Y_n - \mu_n| > |1 - \mu_n| + \epsilon) \\
              &\leq \frac{\mathrm{Var}_{X \sim f_{0n}}(Y_n)}{(|1 - \mu_n| + \epsilon)^2} \rightarrow 0,
            \end{align*}
as required.            
\end{proof}

For Gaussian location mixtures, the chi-squared divergence takes a closed form, which is referred to as the Ingster--Suslina method \citep{ingster2012nonparametric}. 
\begin{lemma}\label{lem:chi-square-cal}
	Let $\phi_{\Sigma}$ denote the density function of the $N_p(0,\Sigma)$ distribution for some positive definite $\Sigma \in \mathbb{R}^{p \times p}$. Define $f_0:=\phi_\Sigma$ and $f_1(\cdot):=\int_{\mathbb{R}^p} \phi_\Sigma(\cdot-\mu) \, d\nu(\mu)$ for some distribution $\nu$ on $\mathbb{R}^p$. Then, for any measurable $A\subseteq\mathbb{R}^p$, we have
	\begin{align*}
	\chi^2(f_1\|f_0) &= \mathbb{E}_{(\mu_1,\mu_2)\sim\nu\otimes\nu}\exp\left(\mu_1^T\Sigma^{-1}\mu_2\right)-1, \\
	\int_A\frac{f_1^2}{f_0} &= \mathbb{E}_{(\mu_1,\mu_2)\sim\nu\otimes\nu}\bigl\{\exp\left(\mu_1^T\Sigma^{-1}\mu_2\right)\mathbb{P}_{x\sim N_p(\mu_1+\mu_2,\Sigma)}(x\in A)\bigr\}.
	\end{align*}
\end{lemma}
\begin{proof}
	By Fubini's theorem, we have
	\begin{align*}
	\int_{A} \frac{f_1^2}{f_0} &= \mathbb{E}_{x\sim \phi_{\Sigma}}\biggl(\frac{f_1^2(x)}{f_0^2(x)}\mathbbm{1}_{\{x\in A\}}\biggr)  \\
	&= \mathbb{E}_{x\sim\phi_\Sigma} \biggl(\frac{\bigl\{\mathbb{E}_{\mu\sim \nu}\phi_{\Sigma}(x-\mu)\bigr\}^2}{\phi^2_{\Sigma}(x)}\mathbbm{1}_{\{x\in A\}}\biggr) \\
	&=                                                  \mathbb{E}_{x\sim\phi_\Sigma}\biggl\{\mathbb{E}_{(\mu_1,\mu_2)\sim\nu\otimes\nu}\Bigl(e^{-(\norm{\mu_1}^2_{\Sigma^{-1}}+\norm{\mu_2}^2_{\Sigma^{-1}})/2 + \langle x,\mu_1+\mu_2 \rangle_{\Sigma^{-1}}}\mathbbm{1}_{\{x\in A\}}\Bigr)\biggr\} \\
	&= \mathbb{E}_{(\mu_1,\mu_2)\sim\nu\otimes\nu}\biggl\{\mathbb{E}_{x\sim\phi_\Sigma}\Bigl(e^{-(\norm{\mu_1}^2_{\Sigma^{-1}}+\norm{\mu_2}^2_{\Sigma^{-1}})/2+\langle x,\mu_1+\mu_2 \rangle_{\Sigma^{-1}}}\mathbbm{1}_{\{x\in A\}}\Bigr)\biggr\}\\
	&= \mathbb{E}_{(\mu_1,\mu_2)\sim\nu\otimes\nu}\bigl\{\exp\left(\mu_1^T\Sigma^{-1}\mu_2\right)\mathbb{P}_{x\sim N_p(\mu_1+\mu_2,\Sigma)}(x\in A)\bigr\},
	\end{align*}
as requried.  Taking $A=\mathbb{R}^p$, we also have
	$$\chi^2(f_1\|f_0)=\int_{\mathbb{R}^p} \frac{f_1^2}{f_0}-1=\mathbb{E}_{(\mu_1,\mu_2)\sim\nu\otimes\nu}\exp\left(\mu_1^T\Sigma^{-1}\mu_2\right)-1.$$
	The proof is complete.
\end{proof}
The following lemma follows immediately from the proof of \citet[][Theorem~2.1]{chen2018robust}.
\begin{lemma}\label{lem:median}
	Consider independent observations $X_j\sim N(\theta + \delta_j,\sigma^2)$ for $j\in[p]$ and the estimator $\wh{\theta}:=\text{\sf Median}(X_1,\ldots,X_p)$. Assume that $\sum_{j=1}^p \mathbbm{1}_{\{\delta_j\neq 0\}}\leq s \leq p/4$. Then there exist universal constants $C_1,C_2,C_3>0$, such that
	$$|\wh{\theta}-\theta|\leq C_1\sigma\biggl(\frac{s}{p} + \sqrt{\frac{1+x}{p}}\biggr),$$
	with probability at least $1-e^{-C_2x}$ for any $x>0$ such that $C_3\left(\frac{s}{p} + \sqrt{\frac{1+x}{p}}\right)\leq 1$.
\end{lemma}

\section*{Acknowledgements}

The research of Chao Gao was supported in part by NSF grant DMS-1712957 and NSF CAREER Award DMS-1847590.  The research of Richard J. Samworth was supported by an EPSRC fellowship and an EPSRC Programme Grant.

\begin{small}
  \bibliographystyle{dcu}
\bibliography{reference}
\end{small}


\end{document}